\documentclass[letterpaper,12pt]{article}
\usepackage[letterpaper, total={6.5in, 9in}]{geometry}
\usepackage[utf8]{inputenc}
\usepackage{amsfonts}
\usepackage{amsthm}
\usepackage{amssymb}
\usepackage{mathtools}
\usepackage{amsmath}
\usepackage{relsize}
\usepackage{comment}
\usepackage{hyperref}
\usepackage{graphicx}
\usepackage{relsize}
\usepackage{subcaption}
\usepackage[labelformat=parens,labelsep=quad,skip=3pt]{caption}
\usepackage{float}
\usepackage{xcolor}

\usepackage[T1]{fontenc}
\usepackage{authblk}

\usepackage[ backend=bibtex, style=ieee, sorting=nty  ]{biblatex}
\DeclareFieldFormat{journaltitle}{#1\isdot}
\DeclareFieldFormat[article,book]{title}{#1\isdot}
\DeclareFieldFormat[book]{title}{#1\addcomma}

\usepackage[normalem]{ulem}

\newtheorem{thm}{Theorem}[section]

\newtheorem{cor}{Corollary}[section]
\newtheorem{lem}{Lemma}[section]

\author{Hossein Dabirian\thanks{dabirian@umich.edu}}
\author{Vijay Subramanian\thanks{vgsubram@umich.edu}}
\affil{EECS Department, University of Michigan}

\date{}

\allowdisplaybreaks

\begin{document}
\title{Connectivity of a Family of Bilateral Preference Random Graphs}
\maketitle

\begin{abstract}
    We study the bilateral preference graphs $\mathit{LK}(n, k)$ of La and Kabkab, obtained as follows. Put independent and uniform $[0,1]$ weights on the edges of the complete graph $K_n$. Then, each edge $(i,j)$ is included in $\mathit{LK}(n,k)$ if it is bilaterally preferred, in the sense that it is among the $k$ edges of lowest weight incident to vertex $i$, and amongst the $k$ edges of lowest weight incident to vertex $j$. We show that $k = \log(n)$ is the connectivity threshold, solving a conjecture of La and Kabkab, and obtaining finer results about the window. We also investigate the asymptotic behavior of the average degree of vertices in $\mathit{LK}(n, k)$ as $n\rightarrow\infty$.
\end{abstract}

\section{Introduction}\label{sec:intro}

In this work we study the bilateral preference graphs $\mathit{LK}(n,k)$ proposed by La and Kabkab (\textsf{LK})~\cite{la2015new}, and prove their connectivity conjecture. These graphs are subgraphs of the complete graph $K_n$, and are constructed as follows. Assign \emph{i.i.d.} uniform in $[0,1]$ weights to the $\binom{n}{2}$ edges. Then, edge $(i,j)$ is present in $\mathit{LK}(n,k)$ if there is bilateral agreement in the preferences of vertices $i$ and $j$ for each other---if it is among the $k$ edges of lowest weight incident to vertex $i$, and amongst the $k$ edges of lowest weight incident to vertex $j$.  Due to the bilateral agreement in the preferences, these graphs differ from both the Cooper--Frieze (\textsf{CF}) $\mathit{CF}(n,k)$ random graphs~\cite{cooper1995connectivity} that are constructed via unilaterally preferred edges, and the (homogeneous or inhomogeneous) Erd\H{o}s--R\'enyi (\textsf{ER}) $\mathit{G}(n,p)$ random graphs~\cite{erdHos1959randomgraph,erdHos1960evolution,renyi1959connected,shang2023connectivity} where the edges are chosen independently.

Our main results---Theorem~\ref{thm:disconnected} and Theorem~\ref{thm:connected}---can be collectively summarized as follows.

\begin{thm}\label{thm:mainres1}
   The random graph $\mathit{LK}(n,k)$ with $k \leq \log (n) + t'\log \log (n)\sqrt{\log (n)}$ for $t'<-\frac{1}{2}$, is not connected with high probability, and when $k \geq \log (n) + t'\log \log (n)\sqrt{\log (n)}$ for $t'>\frac{1}{2}$, it is connected with high probability.
\end{thm}

\paragraph{Outline of the Proof of the Main Results:} The proof of the disconnectedness result---Theorem~\ref{thm:disconnected}---is established using the second moment method in which we prove that the expected number of isolated vertices grows to infinity and show that asymptotically any pair of vertices is isolated in an independent fashion. From this, we show that with high probability there exists an isolated vertex which yields the disconnectedness result. For the connectivity part---Theorem~\ref{thm:connected}---, we take the same approach as in La and Kabkab~\cite{la2015new}---namely, find 
\textsf{ER} sub and super random graphs with $n$ vertices of the $\mathit{LK}(n,k)$ graph. This analysis then works by excluding the possibility of having components with $r$ vertices for $10\leq r\leq n/2$. Independently, we also rule out the existence of components of size $O(1)$. Note that the connectivity/disconnectedness results hold beyond the regime where one can develop a simple association with \textsf{ER} random graphs---as in La and Kabkab~\cite{la2015new}. Whereas we don't present the proof here, using the methodology of Cooper and Frieze~\cite{cooper1995connectivity} one gets sharper connectivity/disconnectedness results in comparison to La and Kabkab~\cite{la2015new}, but the final result is still not as sharp as as Theorems~\ref{thm:connected} and \ref{thm:disconnected}.

In addition to the results on connectivity of $\mathit{LK}(n,k)$ graphs, we also present results on the asymptotic behavior of the mean degree---see Theorem~\ref{thm:genasymptotic} and Theorem~\ref{thm:meandegree}---, which we summarize below.

\begin{thm}\label{thm:mainres2}
For $\mathit{LK}(n,k)$ graphs, suppose $D$ represents the degree of a randomly chosen vertex. If $k=o(\sqrt{n})$, then we have the following asymptotic characterization for the mean degree of $\mathit{LK}(n,k)$ graphs,
\begin{align*}
    \mathbb{E}[D]=k - \left[ \frac{(2k-1)}{2^{2k-1}}\binom{2k-2}{k-1} \right](O(k^2/n)+1).
\end{align*}
If we further assume that $k=o(n^{1/3})$ and $k$ grows to infinity as $n$ goes to infinity, then the average degree has the following asymptotic behavior,
\[\mathbb{E}[D]= k - \sqrt{\frac{k}{\pi}} + \frac{1}{8\sqrt{\pi k}} + o\left(\frac{1}{\sqrt{k}}\right).\]
\end{thm}

\smallskip
\noindent\textbf{Remark:}
   Theorem~\ref{thm:disconnected} and Theorem~\ref{thm:connected} in combination with Theorem~\ref{thm:meandegree} show that the \textsf{LK} random graph family provides another instance (except for a small gap around $\log(n)$) wherein, on the one hand, the mean degree being strictly greater than $\log(n)$ implies connectivity, and, on the other hand, the mean degree being strictly smaller that $\log(n)$ implies the existence of isolated vertices. For the \textsf{LK} family this holds despite the dependencies in the graph construction, unlike either the homogeneous or inhomogeneous \textsf{ER} random graph family. Finally, note also that the asymptotic characterization for the mean degree of $\mathit{LK}(n,k)$ graphs spans a wider range of $k$ when compared to the results of Moharrami et al.~\cite{moharrami2024erlang} (to be discussed later). Note also that our precise characterization of the mean degree yields a correction to the affine behavior conjectured in \cite[Section 6.4]{la2015new}.

\smallskip
\noindent\textbf{Open Problems:} Whereas Theorem~\ref{thm:disconnected} and Theorem~\ref{thm:connected} establish the conjecture by La and Kabkab, they still do not precisely identify the connectivity threshold in comparison to existing results on the (homogeneous) \textsf{ER} model~\cite{renyi1959connected} or even the \textsf{CF} model~\cite{cooper1995connectivity}---specifically, if $\log (n) -0.5\log \log (n)\sqrt{\log (n)} \leq k \leq \log (n) + 0.5\log \log (n)\sqrt{\log (n)}$, then it is not known if connectivity holds or not for the \textsf{LK} model. We believe this arises due to our use of analytically simpler sufficient conditions to prove our results---for example, to show that vertex $i$ is isolated, we insist that all of its edges in $\mathbb{K}_n$ are not preferred by its top $k$ neighbors, with no intersection between the preferred edges of the neighbors. We use an elaborate sufficient condition to prove connectivity. Hence, a more elaborate (or even exact) set of sufficient conditions for both regimes could close the gap---this is an open problem for future work. Similarly, the proof methodology used to establish connectivity in \cite{la2015new} also allows La and Kabkab to determine the diameter of the $\mathit{LK}(n,k)$ graphs, so another open problem is characterizing the diameter for the additional range of $k(n)$ from Theorem~\ref{thm:connected} where connectivity holds. Finally, two additional question are worth investigating. The first is to characterize the limiting distribution of $(D-\mathbb{E}[D])/\sqrt{\mathrm{Var}[D]}$. We expect this to be hard to answer owing to the complicated dependence structure of the graph. Note that even $\mathrm{Var}[D]$ is hard to calculate---using our method in the proof of Theorem~\ref{thm:meandegree} to calculate is hard due to the dependence of the connectivity of the nodes. The second is to find conditions under which Theorem~\ref{thm:meandegree} can be generalized to a concentration of node degrees around the mean degree, which is a result that holds for \textsf{ER} random graphs~\cite[Chapter 4]{draief2010epidemics} when $p(n)=\tfrac{t \log(n)}{n}$ (for $t>0$ and all large enough $n$).

\smallskip
\noindent 
\textbf{Organization of Paper:} We start by discussing related work in Section~\ref{subsec:RW}. Thereafter, we formalize the mathematical model in Section~\ref{sec:model}. In Section~\ref{sec:asympdist} we show how the multinomial distribution arises in the limit of $n\rightarrow\infty$ via a connection to an infinite urn model. Sections~\ref{sec:disconnectivity}, \ref{sec:connectivity} and \ref{sec:avgdegree}, then establish the connectivity results and the characterization of the mean degree of the graph, respectively.

\subsection{Related Work}\label{subsec:RW}


Graphs have been used to study interesting phenomena in many application domains~\cite{easley2010networks,jackson2010social,newman2018networks}. They have been studied using probabilistic tools for about $50$ years~\cite{draief2010epidemics,van2016random}---starting with Gilbert in \cite{gilbert1959random}, and Erd\H{o}s and R\'enyi in \cite{erdHos1959randomgraph, erdHos1960evolution}. 
An (homogeneous) Erd\H{o}s--R\'enyi random graph with $n$ vertices is characterized by a parameter $0\leq p(n) \leq 1$ that is the probability of existence of each potential undirected edge between two vertices. A realization of a graph from the \textsf{ER}-model is denote by $\mathit{G}(n,p)$ where the dependence of $p$ on $n$ is suppressed for brevity. Edges then appear independently based on Bernoulli coin tosses with probability $p(n)$ for heads, and with no intrinsic preference on the edges on the part of the vertices. 
It was shown~\cite{erdHos1960evolution,erdHos1959randomgraph} that if $p(n)=t \log(n)/n$, with high probability, the graph $\mathit{G}(n,p)$ is connected for $t>1$ and disconnected for $t<1$---the refined result~\cite{renyi1959connected} is that if (for large enough $n$) $p(n)=(\log(n)+c)/n$, then the graph is connected with probability $e^{-e^{-c}}$. The critical property needed for connectivity to hold is that the mean degree $t\log(n)$ is strictly greater than $\log(n)$, and for isolated vertices to exist is that the mean degree is strictly less than $\log(n)$; these results also generalize to inhomogeneous \textsf{ER} graphs---see \cite{shang2023connectivity}.

Cooper and Frieze in \cite{cooper1995connectivity} described a new family of random graphs based on preferences constructed using a distance measure. They considered the complete graph $K_n$ and assigned independent uniform $[0,1]$ random variables as distances to all edges. They, then, keep the $k(n)$ shortest edges incident to each vertex. In this model, the concept of distance induces a preference order on all edges---the shorter an edge is, the higher its preference is. Since the distances are \emph{i.i.d.}, the preference order on all the edges will be chosen uniformly over the set of all permutations. We say a vertex 
proposes an edge if the edge belongs to the set of the $k(n)$ most-preferred edges incident to it. Note that in the Cooper--Frieze model, an edge is kept if at least one of its endpoints 
proposes it, which is a unilateral perspective. A realization of a graph from the \textsf{CF} model is denoted as $\mathit{CF}(n,k)$ with the dependence of $k$ in $n$ suppressed for brevity.

Preference relations of the sort used in the \textsf{CF} model create a complicated structure on any resulting graph. For example, keeping edges based on their ranking even in a unilateral manner, induces an edge dependency in the \textsf{CF} model (which is not the case for the \textsf{ER} model). Then, \cite{cooper1995connectivity}  proved that the graph $\mathit{CF}(n,k)$ is connected with high probability for $k\geq 3$. Cooper and Frieze also provided upper and lower bounds for the probability of connectivity as $n\rightarrow\infty$ when $k=2$. 

Bilateral preference random graphs introduced by La and Kabkab \cite{la2015new} are constructed using the same parameter $k(n)$ as the \textsf{CF} random graphs. 
Considering vertices as agents, all vertices are assumed to have their own preferences on the potential edges with others via a priority or preference order over other vertices, where the individual vertex preferences result from a single global preference order on the $\binom{n}{2}$ possible edges. Then, in the La--Kabkab model, in contrast to both the \textsf{ER} and the \textsf{CF} models, an edge is drawn if and only if each end vertex has the other vertex in its $k(n)$ preferred vertices. In other words, an edge is formed if and only if both end vertices 
propose the edge to the other vertex, which then leads to the need for bilateral preference. Hence, the parameter $k(n)$ is now the \textit{maximum} number of other vertices each vertex wants as its neighbors. A graph realization from the \textsf{LK} model is denoted by $\mathit{LK}(n,k)$; the dependence of $k$ in $n$ is again suppressed for brevity.

Following \cite{la2015new} one can interpret the bilateral preference based \textsf{LK} model as a network formation process conducted via a game (selfish objective maximization) among bounded rational agents at the vertices. These kinds of network formation processes~\cite{easley2010networks,jackson2010social,newman2018networks} have been studied in social science and economics. We assume that the $n$ agents are aware of their own benefits from the potential pairwise connections with others where the benefit is assessed via an appropriate cost or distance measure---we assume \textit{without loss of generality} that the pairwise costs or distances are generated are independent and either uniformly distributed in $[0,1]$, or exponentially distributed with parameter $1$. We point the reader to \cite[Observation O-3]{la2015new} for the reason why either of these choices of distributions (or any other continuous distribution) results in the same realization of graphs. Again, following \cite{la2015new} the agents are bounded rational, and only use local information instead of global information. Finally, the agents also have limited memory, so that each agent prefers the $k(n)$ most profitable/valuable connections based solely on the cost or distance of the potential edge instead of some global objective like connectivity. Then, each edge (connection) will exist if and only if both parties prefer it, and from this process the graph of pairwise connections will be generated. 

A core question studied in \cite{la2015new} is to determine $k(n)$ that results in the graphs produced being connected. La and Kabkab showed that the following results hold with high probability: 1) if $k(n) > C \log(n)$ for $C=2.4625$, then the graph is connected; and 2) $k(n)< c \log(n)$ for $c=0.5$, then the graph has isolated vertices and so is not connected. Furthermore, using extensive simulations, they also conjectured that the connectivity threshold was exactly $k(n)=\log(n)$ which is, surprisingly, the same threshold for connectivity of the \textsf{ER} random graph family in terms of the mean degree. Note once again that $k(n)$ is the maximum degree for an $\mathit{LK}(n,k)$ graph. As discussed earlier, we establish the conjecture of La and Kabkab by providing a finer characterization of when connectivity holds and when it does not.

A different line of work from the connectivity question for $\mathit{LK}(n,k)$ graphs is that of Moharrami \emph{et al.} \cite{moharrami2024erlang}---they introduced a new branching process called the Erlang Weighted Tree (EWT) as the local weak limit of the \textsf{LK} model graphs in the sparse regime. Specifically, in \cite{moharrami2024erlang} the parameter $k(n)$ is a (finite) random parameter $k$ for each vertex that is independently chosen and identically distributed with a distribution on $\mathbb{N}$ with finite mean. In this regime, \textsf{LK} graphs are shown to be different from \textsf{ER} graphs. Moharrami \emph{et al.} show this by finding the degree distribution of the root vertex and also its mean degree, which coincide with the asymptotic degree distribution and mean degree for \textsf{LK} graphs, respectively. As discussed earlier, in Section~\ref{sec:avgdegree} we study the asymptotics of the mean degree of an $\mathit{LK}(n,k)$ graph for a wider range of $k$, and as a consequence of results we present an alternate derivation of the mean-degree from  \cite{moharrami2024erlang} when $k$ is deterministically chosen and finite. Moharrami \emph{et al.} also discuss the probability of extinction of an EWT, and conjecture its relevance to the size of the giant component~\cite{erdHos1960evolution,draief2010epidemics,van2016random}of an \textsf{LK} graph, when there is one.  

\section{Mathematical Model}\label{sec:model}

Consider the complete undirected graph $\mathbb{K}_n$ with vertices $[n]=\{1,2,\cdots,n\}$ and edges $\{(i,j): 1\leq i,j\leq n, i\neq j\}$ where we assume $(i,j)=(j,i)$. Let $k(n)$ be an integer such that $1\leq k(n)\leq n$; henceforth, to avoid cumbersome notation, we will use $k$ instead of $k(n)$. We assign independently and identically distributed random variables called priority scores to all edges of $K_n$. For any explicit calculations, one can assume they are uniformly distributed in $[0,1]$ or exponentially distributed with parameter $1$; this holds because we will only be interested in the order statistics. We denote the score of edge $(i,j)$ by $V(i,j)=V(j,i)$. The set of all scores of the edges of vertex $i\in[n]$ is denoted by $\mathcal{V}_i=\{V(i,1),\cdots,V(i,i-1), V(i,i+1),\cdots, V(i,n)\}$, and all the associated edges by $\mathcal{E}_i=\{(i,1),\cdots, (i,i-1),(i,i+1),\cdots,(i,n)\}$. Without loss of generality, we can also assume that the scores are distinct, then $(R_i^j)_{1\leq j\leq n-1}$ represents an order on $[n]\setminus\{i\}=\{1,2,\cdots,i-1,i+1,\cdots,n\}$ based on $V(i,j)$ values. In other words, for each $1\leq i\leq n$, the random vector $\mathcal{R}_i=(R_i^1,R_i^2,\cdots,R_i^{n-1})$ is a permutation of $[n]\setminus\{i\}$ in which
\[V(i,R_i^1)> V(i,R_i^2)>\cdots> V(i,R_i^{n-1}).\]
As the scores are chosen \emph{i.i.d.} from a continuous distribution, the distribution of the random vector $\mathcal{R}_i$ is uniform among all permutations of $[n]\setminus\{i\}$ as it only depends on the order-statistics. The scores also impose a permutation over all the edges. This plays an important role in defining the bilateral preference \textsf{LK} random graphs.

Let $V(i,j)$ be realized for all edges $(i,j)$ and parameter $k$ be fixed. Then we can determine two different classes of random graphs on vertices $[n]$. We first define the notion of preference. If $V(i,j)$ is among the $k$ largest scores in $\mathcal{V}_i$, i.e., $j\in \mathcal{R}_i^{\leq k}:=\{R_i^1,R_i^2,\cdots,R_i^k\}$ or $(i,j)\in\mathcal{E}_i^{\leq k} := \{(i,R_i^1),\cdots,(i,R_i^k)\}$, we say that vertex $i$ proposes/prefers edge $(i,j)$. Moreover, $\mathcal{E}_i:=\mathcal{E}_i^{\leq n-1}$ denotes all the edges incident to vertex $i$. 
The first model introduced by Cooper and Frieze \cite{cooper1995connectivity} let the edge $(i,j)$ be present if at least one of $i$ or $j$ 
prefers $(i,j)$. Reminder that we denote realizations of these random graphs by $\mathit{CF}(n,k)$, which we call the class of unilaterally proposed graphs. The second model described by La and Kabkab \cite{la2015new} requires a 
preference by both vertices for an edge to appear, which we deem as bilateral preference. Again, we denote realization of this class of graphs by $\mathit{LK}(n,k)$, and we call it the class of bilateral preference random graphs.\footnote{We can also construct \textsf{ER} random graphs through independent node preferences: each node prefers the possible $n-1$ edges independently with probability $\sqrt{p(n)}$ with an edge forming only via bilateral preference.}

Both $\mathit{CF}(n,k)$ and $\mathit{LK}(n,k)$ graphs only depend on the order of $V(i,j)$, and not the precise values. This is the consequence of the scores (costs or distances) being \emph{i.i.d.} from a continuous distribution, which results in a uniformly drawn permutation among all permutations of edges of $\mathbb{K}_n$; again, note \cite[Observation O-3]{la2015new}. Assigning a random variable to each edge only helps to motivate the graph construction, but while determining any underlying probability, we will typically use the random permutations on edges viewpoint.


\section{Asymptotically Equivalent Distribution}\label{sec:asympdist}

In our analysis, we will use the abstraction of an infinite urn model. 
The following two lemmas enable us to connect the probability of an event in the finite permutations space to an event in the infinite urn model. In essence, we will show that the appropriate probabilities converge to a binomial distribution in Lemma \ref{common-lem1} and a negative multinomial distribution in Lemma \ref{common-lem2}. Before delving into details, we will consider the set of all permutations with order restrictions on some elements. The notation $x\succ y$ denotes that $x$ appears earlier than $y$ in the permutations. These simple orders can be combined using AND and OR to create more complex order restrictions on the permutation.
For instance, when we refer to the set of permutations of $\{1,2,3\}$ with order restrictions $1\succ 2 \text{ AND } 1\succ 3$, this narrows down the set of all $3!=6$ permutations to just $\{(1,2,3),(1,3,2)\}$. If we instead use the restriction $1\succ 2 \text{ OR } 1\succ 3$, the resulting set of permutations is $\{(1,2,3),(1,3,2),(2,1,3),(3,1,2)\}$.

\begin{lem}
Assume we have $M=m_0+m_1+\cdots +m_s$ objects, each assigned a type, where $m_i$ denotes the number of objects of type $i$ for each $0\leq i \leq s$. Consider a uniformly random permutation of all $M$ objects with some order restrictions which are only for objects of type 0. Let $X$ represent the number of type 0 objects within the first
$l$ positions of this permutation where $l<M$. The law of $X$ for a given $j<l$ is given by the following formula:
\begin{equation}
    \label{prob-type0}
       \mathbb{P}\left(X=j\right) =  \binom{l}{j} \left(\prod_{i=0}^{j-1} (m_0-i)\right) \left(\prod_{i=0}^{l-j-1} (M-m_0-i) \right) 
       \frac{(M-l)!}{M!}.
\end{equation}
We further assume there exists a parameter $n$ such that $s,l=o(n^{1/4})$, and there are constants $h_i>0$ with $m_i=nh_i+\epsilon_i$ where $\sum_{i=1}^s|\epsilon_i|=o(n^{1/4})$ and $\max_{1\leq i\leq s}h_{s_i}=O(1)$. Under these conditions, the asymptotic probability is given by:


\begin{equation}
     \mathbb{P}\left(X=j\right)=\binom{l}{j}\frac{h_0^j(h_1+\cdots+h_s)^{l-j}}{(h_0+h_1+\cdots+h_s)^l}\left(1+o(n^{-1/2})\right),
     \label{prob-type0-asy}
\end{equation}
as $n\rightarrow\infty$. 
\label{common-lem1}
\end{lem}
\begin{proof}
To compute this probability we will count the number of permutations with $j$ elements of type 0 at the first $l$ place then divide it by the number of all permutations. Any given order restrictions on objects of type 0 make both the numerator and denominator of this fraction divided by the number of symmetries. As a result, we can assume there is no order restriction on objects of type 0.

The first part is straightforward by choosing those $j$ places at the first $l$ observations for objects of type 0 and $m_0 - j$  places in the remaining part. Next, by counting the number of desired arrangements for objects of type 0 and other types we arrive at the following formula
\[\mathbb{P}\left(X=j\right) = \frac{\binom{l}{j}\binom{M-l}{m_0 - j}m_0!(M-m_0)!}{M!}.\]

Now \eqref{prob-type0} can be derived through straightforward calculations.
For the second part, we first simplify $\frac{(M-l)!}{M!}$ to obtain
\[\binom{l}{j}\left(\frac{\prod_{i=0}^{j-1} (m_0-i) \prod_{i=0}^{l-j-1}(M-m_0-i)}{\prod_{i=0}^{l-1} (M-i)}\right).\]

Note that $m_0-i = h_0 n+o(n^{1/4})$ since $i\leq j < l = o(n^{1/4})$. Similarly, $M-m_0-i=(h_1+\cdots h_s)n+o(n^{1/4})$ and $M-i=(h_0+h_1+\cdots+h_s)n+o(n^{1/4})$ for each $i$. Substituting these expressions in the second term in the product above results in
\[\frac{\left(h_0 n + o(n^{1/4})\right) ^ j \left((h_1+\cdots +h_s)n+ o(n^{1/4})\right)^{l-j}}{\left((h_0+h_1+\cdots +h_s)n+ o(n^{1/4})\right)^{l}}\]
To derive \eqref{prob-type0-asy}, we factor out $n$ from each term within the parentheses and simplify further, using the fact that $j,l-j=o(n^{1/4})$
\begin{align*}
        \frac{\left(h_0  + o(n^{-3/4})\right) ^ j \left(h_1+\cdots +h_s+ o(n^{-3/4})\right)^{l-j}}{\left(h_0+h_1+\cdots +h_s+ o(n^{-3/4})\right)^{l}} & = \left( \frac{h_0^j(h_1+\cdots+h_s)^{l-j}}{(h_0+h_1+\cdots+h_s)^l}\right) \\
        & \quad \times \frac{\left(1 + \frac{o(n^{-3/4})}{h_0}\right) ^ j \left(1+ \frac{o(n^{-3/4})}{h_1+\cdots+h_s}\right)^{l-j}}{\left(1+ \frac{o(n^{-3/4})}{h_0+h_1+\cdots+h_s}\right)^{l}}\\
& = \left( \frac{h_0^j(h_1+\cdots+h_s)^{l-j}}{(h_0+h_1+\cdots+h_s)^l}\right) \left(1+o(n^{-3/4})\right)^{2l}\\
        & = \left( \frac{h_0^j(h_1+\cdots+h_s)^{l-j}}{(h_0+h_1+\cdots+h_s)^l}\right) \left(1+o(n^{-3/4})\right)^{o(n^{1/4})}\\
        & = \left( \frac{h_0^j(h_1+\cdots+h_s)^{l-j}}{(h_0+h_1+\cdots+h_s)^l}\right) \left(1+o(n^{-1/2})\right).
\end{align*}
This completes the proof.
\end{proof}

\begin{lem}
Under the same assumptions as Lemma \ref{common-lem1}, consider $M$ objects, each assigned a type from $0$ to $s$, with $m_i$ objects of type $i$ for each $0\leq i\leq s$. Let $X=(i_1,\cdots,i_s)$ be a random vector where $i_j$  represents the number of occurrences of objects of type $j$ before the first occurrence of objects of type 0 for $1\leq j\leq s$ in a uniformly random permutation. Then the law for $X$ is given by:
\begin{align}
    \mathbb{P}\left(X=(i_1,\cdots,i_s)\right) 
    & =   \frac{m_0\binom{m_1}{i_1}\cdots\binom{m_s}{i_s}(i_1+\cdots+i_s)!(M-i_1-\cdots-i_s-1)!}{M!}.
    \label{prob-vector}
\end{align}
Similar to Lemma \ref{common-lem1}, if we further assume $s=o(n^{1/4})$ and, $i_j<m_j=nh_j+\epsilon_j$, with $\sum_{j=1}^s|\epsilon_j|, \sum_{j=1}^s i_j=o(n^{1/4})$, then the probability \eqref{prob-vector} asymptotically approaches:
\[\frac{h_0 h_1 ^{i_1}\cdots h_s^{i_s} }{( h_0+h_1+\cdots+h_s)^{i_1+\cdots+i_s+1}}\cdot\frac{(i_1+i_2+\cdots+i_s)!}{i_1!i_2!\cdots i_s!} \left(1+o(n^{-1/2})\right),\]
as $n\rightarrow\infty$.
\label{common-lem2}
\end{lem}
\begin{proof}

 It is easy to check that the number of permutations of $M$ objects with exactly $i_j$ objects of type $j$ at the first $i_1+i_2+\cdots+i_s$ places for each $1\leq j\leq s$, and with the first object of type 0 at the $(i_1+\cdots+i_s+1)$-th place, is as follows
\[m_0\binom{m_1}{i_1}\cdots\binom{m_s}{i_s}(i_1+\cdots+i_s)!(M-i_1-\cdots-i_s-1)!.\]
Dividing it by the total number of permutations, i.e., $M!$ yields \eqref{prob-vector}.

Then, note that
\begin{align*}
    \binom{m_j}{i_j}& =\frac{m_j(m_j-1)\cdots(m_j-i_j+1)}{i_j!}\\
    & =\frac{\left(nh_j+o(n^{1/4})\right)^{i_j}}{i_j!}\\
    & =\frac{n^{i_j}h^{i_j}\left(1+o(n^{-3/4})\right)^{i_j}}{i_j!}.
\end{align*}
Moreover, using $s=o(n^{1/4})$, $\sum_{j=1}^s|\epsilon_j|=o(n^{1/4})$, and $\sum_{j=1}^s|i_j|=o(n^{1/4})$, we have
\begin{align*}
    \frac{(M-i_1-\cdots-i_s-1)!}{M!} &=\frac{1}{M(M-1)\cdots(M-i_1-\cdots i_s)} \\
    &=\frac{1}{\left((h_0+h_1+\cdots +h_s)n + o(n^{1/4})\right)^{i_1+\cdots+i_s+1}}\\
    & = \frac{1}{\left((h_0+\cdots+h_s)n\right)^{i_1+\cdots+i_s+1}}\cdot\frac{1}{(1+\frac{1}{h_0+\cdots+h_s}\cdot o(n^{-3/4}))^{i_1+\dots +i_s+1}}\\
    & = \frac{1}{(h_0+\cdots+h_s)^{i_1+\cdots+i_s+1}n^{i_1+\cdots+i_s+1}} \cdot \frac{1}{(1+o(n^{-3/4}))^{i_1+\dots +i_s+1}}.
\end{align*}
 We substitute the last two asymptotic expressions along with $m_0=h_0n(1+o(n^{-3/4}))$ into \eqref{prob-vector} to estimate the probability of the event in question as follows:
\begin{align*}
        \mathbb{P}\left(X=(i_1,\cdots,i_s)\right) & = \left(\frac{h_0 h_1^{i_1}\cdots h_s^{i_s}(i_1+\cdots+i_s)!}{( h_0+h_1+\cdots+h_s)^{i_1+\cdots+i_s+1}i_1!\cdots i_s!}\right)\frac{\left(1+o(n^{-3/4})\right)^{i_1+\cdots+i_s+1}}{\left(1+o(n^{-3/4})\right)^{i_1+\cdots+i_s+1}}\\
        & = \left(\frac{h_0 h_1^{i_1}\cdots h_s^{i_s}(i_1+\cdots+i_s)!}{( h_0+h_1+\cdots+h_s)^{i_1+\cdots+i_s+1}i_1!\cdots i_s!}\right)\left(1 + o(n^{-3/4})\right)^{2(i_1+\cdots+i_s+1)}\\
        & = \left(\frac{h_0 h_1^{i_1}\cdots h_s^{i_s}(i_1+\cdots+i_s)!}{( h_0+h_1+\cdots+h_s)^{i_1+\cdots+i_s+1}i_1!\cdots i_s!}\right)\left(1 + o(n^{-3/4})\right)^{o(n^{1/4})}\\
        & = \left(\frac{h_0 h_1^{i_1}\cdots h_s^{i_s}(i_1+\cdots+i_s)!}{( h_0+h_1+\cdots+h_s)^{i_1+\cdots+i_s+1}i_1!\cdots i_s!}\right)\left(1 + o(n^{-1/2})\right).
\end{align*}
This finishes the proof.
\end{proof}

Lemmas \ref{common-lem1} and \ref{common-lem2} relate certain statistics in the space of uniformly random permutations of a large set of objects, each assigned a specific type, to the binomial and negative multinomial distributions. Specifically, under certain assumptions on the number of types and objects of each type, Lemma \ref{common-lem1} establishes a relationship between the probability distribution of the number of occurrences of a particular type in the first $l$ positions of this permutation (for sufficiently small $l$) and a binomial distribution.  Additionally, Lemma \ref{common-lem2} links the number of occurrences of non-zero types before the first occurrence of type $0$ to a negative multinomial distribution. As mentioned earlier, this distributional convergence will help us greatly in our analysis. The special case required in this paper fixes all $h_i\equiv1$. Therefore, we state the following lemma.
\begin{lem}
Under the same assumptions of Lemma \ref{common-lem1}, there are $M$ objects, each assigned a type from $0$ to $s$, with $m_i$ objects of type $i$ for $0\leq i\leq s$. We similarly assume the existence of a parameter $n$ with $s=o(n^{1/4})$, $m_i=n+\epsilon_i$ so that $\sum_{j=1}^s|\epsilon_j|=o(n^{1/4})$. Additionally, let $l=o(n^{1/4})$ grow to infinity when $n\rightarrow\infty$ under the extra condition that $s=o(l)$.

Denoting by $X$ the number of objects of type $0$ within the first $l$ positions of this permutation, we have the following bound for its lower tail
\[\mathbb{P}\left( X\leq t \right) \leq \exp\left(\frac{-l}{s+1}+ t\log\left(\frac{l}{ts}\right)+ t + \log (t+1) \right)\big(1+o(n^{-1/2})\big),\] 
where $t\leq l/s$. In particular, if there are parameters $k,a,b$ with $k=o(n^{1/6})$ growing to infinity as $n\rightarrow\infty$, $a,b=O(1)$, and $s,t=\sqrt{k}+O(1)$ so that $l=(s-b)(k-s-a)$, for sufficiently large $n$, then we have
\[\mathbb{P}\left( X\leq t \right) \leq \exp\left(-k + \frac{1}{2}\sqrt{k}\log(k) + O(\sqrt{k})\right)\] 
\label{common-lem3}
\end{lem}
\begin{proof}
We can assume $k >s$ for sufficiently large $n$. It follows from $t\leq l/s$ that $\binom{l}{j} \frac{s^{l-j}}{(s+1)^l}$ is increasing for $0\leq j\leq t$. Using this fact and Lemma \ref{common-lem1} with $h_i=1$ we have
\begin{align*}
    \mathbb{P}\left( X \leq  t \right) = & \sum_{j=0}^{t} \binom{l}{j} \frac{s^{l-j}}{(s+1)^l} (1+o(n^{-1/2})) \\
    \leq & (t+1) \binom{l}{t} \frac{s^{l-t}}{(s+1)^l}(1+o(n^{-1/2}))\\
    = & \frac{t+1}{s^{t}} \binom{l}{t}\left(1-\frac{1}{s+1}\right)^l(1+o(n^{-1/2}))\\
    = & \frac{t+1}{s^{t}} \binom{l}{t}\left(\Big(1-\frac{1}{s+1}\Big)^{s+1}\right)^{l/(s+1)}(1+o(n^{-1/2}))\\
    \leq &  \frac{t+1}{s^{t}} \binom{l}{t}\exp\left(\frac{-l}{s+1}\right)(1+o(n^{-1/2}))\\
    \leq & \frac{t+1}{s^{t}} \cdot \frac{l^t}{t!}\exp\left(\frac{-l}{s+1}\right)(1+o(n^{-1/2})).
\end{align*}
One can use the bound from Stirling's approximation, i.e., $t!>(t/e)^t$, to get

\begin{align*}
    \mathbb{P}\left( X \leq t \right)  \leq & \frac{(t+1)e^t l^t}{s^{t} t^t} \exp\left(\frac{-l}{s+1}\right)(1+o(n^{-1/2}))\\
    = & \exp\left(\frac{-l}{s+1}+ t\log\left(\frac{l}{ts}\right)+ t +  \log(t+1) \right)(1+o(n^{-1/2})).
\end{align*}

By substituting $s,t=\sqrt{k}+O(1)$ and $l=(s-b)(k-s-a)=k\sqrt{k}+O(k)$ we obtain
\begin{align*}
    \frac{-l}{s+1}+t+\log (t+1)& =\frac{-k\sqrt{k}+O(k)}{\sqrt{k}+O(1)}+2\sqrt{k}+O(\log k)\\
    & = -k+O(\sqrt{k}).\\ 
\end{align*}
Additionally,

\begin{align*}
    t\log\left(\frac{l}{ts}\right) & =\left(\sqrt{k}+O(1)\right)\log\left(\frac{k\sqrt{k}+O(k) 
    }{k + O(\sqrt{k})} \right)\\
    & = \left(\sqrt{k}+O(1)\right)\log\left(\sqrt{k}+O(1) \right)\\
    & = \frac{1}{2}\sqrt{k}\log(k) + O(\sqrt{k}).
\end{align*}

This completes the proof.
\end{proof}
\section{Disconnectedness of $\mathit{LK}(n,k)$ Graphs for $t<1$}\label{sec:disconnectivity}
This section begins by recalling the negative multinomial distribution. In the special case that we need here, we present an urn model to interpret the distribution. Consider a sequence of \emph{i.i.d.} random variables $(X_i)_{i=1}^{\infty}$, each taking $k+1$ possible values $\{0,1,\cdots,k\}$ uniformly at random. We refer to these outcomes as the type of objects, resulting in $k+1$ distinct types. This sequence continues until the first occurrence of an object of type $0$. The probability of observing $i_1$ objects of type 1, $i_2$ objects of type 2, ... , $i_k$ objects of type $k$, before stopping follows the negative multinomial distribution, as described below
\begin{equation}
 f(i_1,i_2,\cdots, i_k)=\frac{1}{(k+1)^{i_1+i_2+\cdots+i_k+1}}\cdot\frac{(i_1+i_2+\cdots+i_k)!}{i_1!i_2!\cdots i_k!},
 \label{negative-multi}
\end{equation}
where $i_j\geq 0$ for $1\leq j\leq k$. We now introduce some Lemmas in order to bound certain probabilities that will appear in the proof.
\begin{lem}
Let $k,n$ be positive integers. With definition (\ref{negative-multi}), we have
\[\mathlarger{\sum}_{\substack{i_1+i_2+\cdots+i_k=n-1\\\forall j: i_j\geq 0}} f(i_1,i_2,\cdots, i_k)=\frac{1}{k+1}\left(\frac{k}{k+1}\right)^{n-1}.\]
\label{binom}
\end{lem}
\begin{proof}
In the urn model under consideration, i.e, $(X_i)_{i=1}^{\infty}$, the first appearance of 0 follows a geometric random variable with parameter $\frac{1}{k+1}$. Therefore, the right-hand side is the probability that this random variable is $n$. The left-hand side sums over the probabilities of all possible numbers of other types observed prior to the $n$-th draw.
\end{proof}


\begin{lem}
Let $k$ be a positive integer and let $0<\delta<1$. Then: 
\[\mathlarger{\sum}_{\substack{k^2(1+\delta)<i_1+i_2+\cdots+i_k<4(k+1)^2 \\ \forall j: i_j\geq 0}} f(i_1,i_2,\cdots, i_k)\geq \mathlarger{\sum}_{\substack{k^2(1+\delta)<i_1+i_2+\cdots+i_k \\ \forall j: i_j\geq 0}} f(i_1,i_2,\cdots, i_k) \; - \; e^{-4(k+1)}.\]
\label{remind}
\end{lem}
\begin{proof}
We will start by bounding the term below which we denote by $r$. We have
\begin{align*}
        r & =  \mathlarger{\sum}_{\substack{k^2(1+\delta)<i_1+i_2+\cdots+i_k \\ \forall j: i_j\geq 0}} f(i_1,i_2,\cdots, i_k) - \mathlarger{\sum}_{\substack{k^2(1+\delta)<i_1+i_2+\cdots+i_k \leq 4(k+1)^2 \\ \forall j: i_j\geq 0}} f(i_1,i_2,\cdots, i_k)\\
        & = \mathlarger{\sum}_{\substack{4(k+1)^2<i_1+i_2+\cdots+i_k \\ \forall j: i_j\geq 0}} f(i_1,i_2,\cdots, i_k)\\
        & = \sum_{n = 4(k+1)^2}^{\infty} \frac{1}{k+1}\left(\frac{k}{k+1}\right)^n\quad \textrm{(Using Lemma \ref{binom})}\\
        & = \left(\frac{k}{k+1}\right)^{4(k+1)^2}\\
        & = \left(\left(1 - \frac{1}{k+1}\right)^{(k+1)}\right)^{4(k+1)}\\
        & \leq e^{-4(k+1)},
\end{align*}
and the result follows.
\end{proof}

\begin{lem}Let $k$ be a positive integer and $0<\delta<1$. Then:
\begin{align}
\begin{split}
    & \mathlarger{\sum}_{\substack{k^2(1+\delta)<i_1+i_2+\cdots+i_k<4(k+1)^2 \\ \forall j,\, i_j\geq k}}f(i_1,i_2,\cdots, i_k) \\ & \qquad \qquad \qquad \qquad \geq (1 - k e^{\frac{-k\delta^2}{5}}) \mathlarger{\sum}_{\substack{k^2(1+\delta)<i_1+i_2+\cdots+i_k<4(k+1)^2 \\ \forall j,\, i_j\geq 0}}f(i_1,i_2,\cdots, i_k).
\end{split}\label{eq:bern}
\end{align}
\label{bern}
\end{lem}
Before proving the result we point out the subtle difference in the summations on both sides of \eqref{eq:bern}---the LHS considers $i_j \geq k$, and the RHS considers $i_j \geq 0$, and so, the RHS summation is larger than the LHS summation. Hence, the result is non-trivial.
\begin{proof}
Let $N$ denote the step in which we first observe a type 0 object in the sequence $(X_i)_{i=1}^{\infty}$. In other words, $X_N=0$ and for all $i< N$, $X_i\neq 0$. Moreover, we let $I_j$ be the number of occurrences of objects of type $j$ before step $N$. At each step prior to the $N$-th observation, whether an outcome is of type $m$ or not, i.e., $X_j=m$, has a Bernoulli distribution with parameter $\tfrac{1}{k}$. Denote this random variable by $X_j^{(m)}$ where $j$ refers to the step number. For an arbitrary $1\leq m \leq k$, we have $\mathbb{E}[X_j^{(m)}|j<N]=\tfrac{1}{k}$ and $\mathrm{Var}[X_j^{(m)}|j<N]=\tfrac{k-1}{k^2}$. If we assume $N=n>k^2(1+\delta)$, then the assumption $I_m<k$ implies
\[\frac{X_1^{(m)}+\cdots+X_{n-1}^{(m)}}{n-1}<\frac{k}{k^2(1+\delta)}=\frac{1}{k(1+\delta)}.\]
However, the expected value of the time average above given $N=n$ is $\frac{1}{k}$. Hence, the difference $\tfrac{1}{k(1+\delta)}-\tfrac{1}{k}=-\tfrac{\delta}{k(1+\delta)}$, indicates a lower tail event. To analyze this event, we apply the Bernstein inequality, \cite[Equation (2.10)]{boucheronlugosimassartbook2013}, for $n>k^2(1+\delta)$:
\begin{align*}
    \mathbb{P}\left(I_m<k|N=n\right)  & \leq \mathbb{P}\left(\frac{X_1^{(m)}+\cdots+X_{n-1}^{(m)}}{n-1} < \frac{1}{k(1+\delta)}\; \middle |\; N=n\right)\\
      & \leq\exp\left(\frac{-(n-1)\times \frac{\delta^2}{k^2(1+\delta)^2}}{\frac{2(k-1)}{k^2}+\frac{2\delta}{3k(\delta+1)}}\right) \\
       & \leq\exp\left(\frac{-k^2(1+\delta)\times \frac{\delta^2}{k^2(1+\delta)^2}}{\frac{2(k-1)}{k^2}+\frac{2\delta}{3k(\delta+1)}}\right) \\
       & = \exp\left(\frac{-3k^2\delta^2}{6(k-1)(\delta+1) + 2\delta k}\right)\\
       & \leq \exp\left(\frac{-k\delta^2}{5}\right).
\end{align*}
Using the union bound yields 
\[\mathbb{P}\left(\forall j:I_j\geq k\bigg|N=n\right) \geq 1 - k\exp\left(\frac{-k\delta^2}{5}\right),\]
for $n>k^2(1+\delta)$. Therefore, we can apply this bound and definition \eqref{negative-multi} to obtain
\begin{align*}
    & \mathlarger{\sum}_{\substack{k^2(1+\delta)<i_1+i_2+\cdots+i_k<4(k+1)^2 \\ \forall j,\, i_j\geq k}}f(i_1,i_2,\cdots, i_k) \\ & = 
      \mathlarger{\sum}_{ k^2(1+\delta)< n < 4(k+1)^2}\;\mathlarger{\sum}_{\substack{\sum i_m=n-1,\\ \forall j,\, i_j\geq k}}f(i_1,i_2,\cdots, i_k) \\ 
    & = 
     \mathlarger{\sum}_{ k^2(1+\delta)< n < 4(k+1)^2}\mathbb{P}\left(\forall j:I_j\geq k, N=n\right) \\ 
     & = 
     \mathlarger{\sum}_{ k^2(1+\delta)< n < 4(k+1)^2}\mathbb{P}\left(N=n\right)\mathbb{P}\left(\forall j:I_j\geq k\big| N=n\right) \\ 
     & \geq \left(1 - k\exp\left(\frac{-k\delta^2}{5}\right)\right) \mathlarger{\sum}_{ k^2(1+\delta)< n < 4(k+1)^2}\mathbb{P}\left(N=n\right)\\
    & = \left(1 - k\exp\left(\frac{-k\delta^2}{5}\right)\right)\mathlarger{\sum}_{\substack{k^2(1+\delta)<i_1+i_2+\cdots+i_k<4(k+1)^2 \\ \forall j,\, i_j\geq 0}}f(i_1,i_2,\cdots, i_k).
\end{align*}
This completes the proof.
\end{proof}
\begin{lem}
Let $k$ be a positive integer, and let $0\leq \delta\leq 1$ depend on $k$ such that for sufficiently large $k$ we have $\delta\geq \sqrt{\tfrac{6\log(k)}{k}}$. Then, the following asymptotic relationship holds 
\begin{equation}\label{eq:fbound}
\mathlarger{\sum}_{\substack{k^2(1+\delta)<i_1+i_2+\cdots+i_k<4(k+1)^2 \\ \forall j,\, i_j\geq k}}f(i_1,i_2,\cdots, i_k) = \exp\left( \left(-k + \frac{1}{2}\right)(1+\delta) + o(1) \right), \end{equation}
as $k\rightarrow\infty$.
\label{infinite-exp}
\end{lem}
\begin{proof}
First note that Lemma \ref{binom} and the second order approximation $\log(1 - \tfrac{1}{k+1})=-\tfrac{1}{k+1}-\tfrac{1}{2(k+1)^2}+\Theta(\tfrac{1}{k^3})$ imply
\begin{align}
   \mathlarger{\sum}_{\substack{k^2(1+\delta)<i_1+i_2+\cdots+i_k \\ \forall j,\, i_j\geq 0}}f(i_1,i_2,\cdots, i_k)  & = \sum_{N=k^2(1+\delta)+1}^{\infty}\frac{1}{k+1}\left(\frac{k}{k+1}\right)^{N-1}\notag \\
    & = \left(1 - \frac{1}{k+1}\right)^{k^2(1+\delta)}\notag \\
    & = \exp\left(\left(-\frac{1}{k+1}-\frac{1}{2(k+1)^2}+\Theta\left(\frac{1}{k^3}\right)\right)k^2(1+\delta)\right)\notag \\
    & = \exp\left(\left(-k + \frac{1}{2} + \Theta\left(\frac{1}{k}\right) \right) (1 + \delta) \right)\notag \\
    & = \exp\left( (-k + 0.5)(1+\delta) + \Theta\left(\frac{1}{k}\right) \right).
    \label{negative-multi-sum}
\end{align}

We can now simply upper bound the left-hand side of \eqref{eq:fbound} using \eqref{negative-multi-sum} as follows:
\begin{align}
    & \mathlarger{\sum}_{\substack{k^2(1+\delta)<i_1+i_2+\cdots+i_k<4(k+1)^2 \\ \forall j,\, i_j\geq k}}f(i_1,i_2,\cdots, i_k)  \notag \\
   & \leq 
   \mathlarger{\sum}_{\substack{k^2(1+\delta)<i_1+i_2+\cdots+i_k \\ \forall j,\, i_j\geq 0}}f(i_1,i_2,\cdots, i_k) \notag \\ 
   & = \exp\left( (-k + 0.5)(1+\delta) + \Theta\left(\frac{1}{k}\right) \right) \notag\\
   & = \exp\left( (-k + 0.5)(1+\delta) + o(1) \right). 
   \label{negative-multi-sum-side1}
\end{align}

We next use Lemmas \ref{bern}, \ref{remind}, and \eqref{negative-multi-sum} to lower bound the left-hand side:
\begin{align*}
       & \mathlarger{\sum}_{\substack{k^2(1+\delta)<i_1+i_2+\cdots+i_k<4(k+1)^2 \\ \forall j,\, i_j\geq k}}f(i_1,i_2,\cdots, i_k)  \\
       & \geq 
        \Big(1 - k e^{\frac{-k\delta^2}{5}}\Big) \mathlarger{\sum}_{\substack{k^2(1+\delta)<i_1+i_2+\cdots+i_k<4(k+1)^2 \\ \forall j,\, i_j\geq 0}}f(i_1,i_2,\cdots, i_k)\\ 
    & \geq
        \Big(1 - ke^{\frac{-k\delta^2}{5}}\Big)\left(\mathlarger{\sum}_{\substack{k^2(1+\delta)<i_1+i_2+\cdots+i_k \\ \forall j,\, i_j\geq 0}}f(i_1,i_2,\cdots, i_k)- e^{-4(k+1)}\right) \\
        & = 
        \Big(1 - ke^{\frac{-k\delta^2}{5}}\Big)\left(e^{ (-k + \frac{1}{2})(1+\delta) + \Theta(\frac{1}{k})}-e^{-4(k+1)}\right)\\ 
        & = e^{ (-k + \frac{1}{2})(1+\delta) + \Theta(\frac{1}{k})}\Big(1 - ke^{\frac{-k\delta^2}{5}}\Big)\Big(1 - e^{(\delta-3)k+(\frac{-7+\delta}{2})+\Theta(\frac{1}{k})}\Big)
\end{align*}
The bounds $\sqrt{\frac{6\log k}{k}}\leq \delta \leq 1$ yield $ke^{\frac{-k\delta^2}{5}}\leq \tfrac{1}{\sqrt[5]{k}}\rightarrow 0$, and $e^{(\delta-3)k+(\frac{-7+\delta}{2})+\Theta(\frac{1}{k})}\rightarrow 0$ as $k\rightarrow\infty$. Therefore, we can replace $(1 - ke^{\frac{-k\delta^2}{5}})$ and $(1 - e^{(\delta-3)k+(\frac{-7+\delta}{2})+\Theta(\frac{1}{k})})$ by $e^{o(1)}$ to deduce the following lower bound
\begin{equation}
  \mathlarger{\sum}_{\substack{k^2(1+\delta)<i_1+i_2+\cdots+i_k<4(k+1)^2 \\ \forall j,\, i_j\geq k}}f(i_1,i_2,\cdots, i_k) \geq \exp\left( (-k + 0.5)(1+\delta) + o(1) \right).
    \label{negative-multi-sum-side2}
\end{equation}
Now \eqref{negative-multi-sum-side1} and \eqref{negative-multi-sum-side2} complete the proof.

\end{proof}
We are now in a position to relate the lemmas above to the disconnectedness result. Assume $I_j$ represents the event of vertex $j$ being isolated. We will show that with high probability there exists a $j$ such that $I_j=1$.

\begin{lem}
For a graph of class $LK(n,k)$ when $k=o(\sqrt[8]{n})$ and $k\rightarrow\infty$, we have
\[\mathbb{P}(I_1 = 1)\geq\frac{1}{2}\exp\big( (-k + 0.5)(1+\delta) + o(1) \big), \]
for any $\delta=\delta(k)$ with $ \sqrt{\tfrac{6\log(k)}{k}}\leq \delta(k) <1$.
\label{prob-iso-lower-bound}
\end{lem}
\begin{proof}
As mentioned earlier, any realization of independent random variables $\{V(i,j)\}_{i,j}$ can be regarded as yielding a permutation of all edges $\mathcal{E}$. Without loss of generality, one can assume vertex labels are sorted according to the scores of edges incident to vertex 1. In other words, $V(1,2)>V(1,3)>\cdots>V(1,n)$, or equivalently $R_1^j=j+1$ for $1\leq j\leq n-1$. Then, for vertex 1 to be isolated, it is sufficient that for each $j\in\{2,3,\cdots,k+1\}$, there exist at least $k$ elements of $\mathcal{E}_j$ (except $(1,j)$) appearing before all elements of $\mathcal{E}_1$ in the permutation. Strictly speaking, if $k$ elements of $\mathcal{E}_j$ appear before $(1,j+1)$, then that would be necessary and sufficient. However, we consider a stricter condition insisting that $k$ elements of $\mathcal{E}_j$ appear before all elements of $\mathcal{E}_1$. This ensures that vertex $j$ does not agree on edges $(1,j)$ for $2\leq j \leq k+1$, and vertex $1$ similarly does not agree upon edges $(1,j)$ for $k+2 \leq j \leq n$, thereby isolating vertex 1.

To make this condition even stronger, we insist on no intersections. Define for each $2\leq i\leq k+1$: $\mathcal{E'}_i=\mathcal{E}_i \setminus \{ (i,1) , \cdots, (i,k+1)\}$. As a result, $\{\mathcal{E}_1,\mathcal{E}'_2,\mathcal{E}'_3,\cdots,\mathcal{E}'_{k+1}\}$ is a pairwise disjoint collection with $|\mathcal{E}_1|=n-1, |\mathcal{E}'_i|=n-k-1$. Note that we only determined the order over $\mathcal{E}_1$. Thus, any realization of edge scores induces a uniformly random permutation on $\mathcal{E}_1\cup    \mathcal{E'}_2\cup\cdots\cup\mathcal{E'}_{k+2}$ with the order restriction  $(1,2) \succ (1,3)\succ\cdots\succ (1,n)$ on $\mathcal{E}_1$. 

Let $\mathcal{E}_1$ be objects of type 0 and $\mathcal{E'}_i$ be objects of type $i-1$ for $2\leq i\leq k+2$. This satisfies all the conditions of Lemma \ref{common-lem2}. Therefore, the probability of having $i_j$ elements of type $j$ before the first observation of type 0 when $k^2(\delta + 1) < i_1+\cdots+i_k < 4(k+1)^2=o(n^{1/4})$ can be approximated by the distribution $f(\cdot)$ from  \eqref{negative-multi}. Hence, for sufficiently large $n$ we have: 
\[\mathbb{P}\left(X=(i_1,\cdots,i_k)\right)\geq \frac{1}{2} f(i_1,\cdots,i_k)\]
From what we explained earlier, the additional conditions $i_j\geq k$ for $1\leq j \leq k$ ensure that vertex 1 is isolated. Applying Lemma \ref{infinite-exp} then allows us to conclude that

\[\mathbb{P}(I_1 = 1)\geq\frac{1}{2}\exp\left( (-k + 0.5)(1+\delta) + o(1) \right), \]

which completes the proof.
\end{proof}

\begin{lem}
For an $\mathit{LK}(n,k)$ graph where $k \leq  \log(n) - 3\sqrt{\log(n) \log\log(n)}$, we have
\[\mathbb{E}[\textit{Number of isolated vertices}]\rightarrow \infty\]
In particular, if $k=\lfloor t\log(n) \rfloor $ for $t<1$, the statement is true. 
\label{lem-second-method-1}
\end{lem}
\begin{proof}
Let $\delta = \sqrt{\tfrac{6\log(k)}{k}}$, according to Lemma \ref{prob-iso-lower-bound} for sufficiently large $n$ and $k$, we have
\begin{align*}
    \mathbb{E}[\textit{Number of isolated vertices}] & = \mathbb{E}\left[\sum_{i=1}^n I_i\right]\\
    & = n\mathbb{E}[I_1]\\
    & = n\mathbb{P}(I_1=1)\\
    & \geq \frac{n}{2}\exp\left( (-k + 0.5)(1+\delta) + o(1) \right)\\
    & = \exp\left(\log(n) - k - \sqrt{6k\log(k)}+\frac{1}{2}-\log 2 + o(1)\right)\\
    & \geq  \exp\left(3\sqrt{\log(n) \log\log(n)} - \sqrt{6k\log(k)}\right) \ (\textrm{For sufficiently large } n)\\
    & \geq \exp\left(\left(3-\sqrt{6}\right)\left(\sqrt{\log(n)\log\log(n)}\right)\right),
\end{align*}
wherein the right hand side goes to infinity as $n\rightarrow\infty$.
\end{proof}

To establish the desired result that there is at least one isolated vertex with high probability, it is insufficient to merely show that the mean of a non-negative random variable is unbounded, as this does not directly imply that the probability of it being zero vanishes. Therefore, we utilize the second moment method \cite{alon2016probabilistic}, requiring the analysis of the correlation between pairs of indicator random variables that represent vertex isolation. Let us consider two arbitrary vertices, without loss of generality, vertices 1 and 2. Our goal is to study $\mathbb{P}(I_1=1, I_2=1)=\mathbb{P}(I_1 I_2=1)=\mathbb{E}[I_1 I_2]$. To proceed, we define the following events:

\begin{itemize}
    \item $B_{1;2}:=\{2\in\mathcal{R}^{\leq k}_1\}$, representing the event in which vertex 2 belongs to the $k$-most-preferred set of vertex 1.
    \item $B_{2;1}:=\{1\in\mathcal{R}^{\leq k}_2\}$,  representing the event in which vertex 1 belongs to the $k$-most-preferred set of vertex 2.
    \item $B_{1\cup\,2}:=\{\mathcal{R}^{\leq k}_1\cap\mathcal{R}^{\leq k}_2\neq \phi\}$, identifying the event where the $k$-most-preferred elements of vertices 1 and 2 have intersections. 
    \item  $B_{1\rightarrow2}:=\{\exists i\in\mathcal{R}^{\leq k}_1 \textrm{ such that } (\mathcal{R}^{\leq k}_2 \cup \{2\}) \cap \mathcal{R}^{\leq k}_i \neq \phi\}$. This represents the event where there exists a vertex $i$ among the $k$-most-preferred set of vertex 1 such that 2 or one of $k$-most-preferred set vertices of it belongs to the $k$-most-preferred vertices of $i$. 
    \item $B_{2\rightarrow1}:=\{\exists i\in\mathcal{R}^{\leq k}_2 \textrm{ such that }  (\mathcal{R}^{\leq k}_1 \cup \{1\}) \cap \mathcal{R}^{\leq k}_i \neq \phi\}$. This represents the event where there exists a vertex $i$ among the $k$-most-preferred set of vertex 2 such that 1 or one of $k$-most-preferred set vertices of it belongs to the $k$-most-preferred vertices of $i$. 
    \item $B := B_{1;2}\cup B_{2;1} \cup B_{1\cup\, 2} \cup  B_{1\rightarrow 2} \cup B_{2\rightarrow 1}.$
\end{itemize}

For ease of presentation, we omit the value $1$ from the events $\{I_1=1\}$, $\{I_2=1\}$ or $\{I_1I_2=1\}$. For example, we will use the shorthand $\mathbb{P}(I_1I_2):=\mathbb{P}(I_1I_2=1)$. Furthermore, we assume $k<\log(n)$ for the remainder of this section. We start by bounding the event $B$.

\begin{lem}\label{bound-B}
    We have the following bound for $B$
    \[\mathbb{P}(B)\leq  O(k^3/n).\]
\end{lem}
\begin{proof}
    First, note that $\mathbb{P}(B_{1;2})=\mathbb{P}(B_{2;1})=\frac{k}{n-1}$. Consequently, $\mathbb{P}((B_{1;2} \cup B_{2;1})^c)=1-\mathbb{P}(B_{1;2} \cup B_{2;1})=1-O(k/n)$. Now, given the event $(B_{1;2} \cup B_{2;1})^c=B^c_{1;2}\cap B^c_{2;1}$, the sets  $\mathcal{R}_1^{\leq k} $ and  $\mathcal{R}_2^{\leq k}$ are independent and each is uniformly drawn from the collection of all subsets of size $k$ of $ \{3, \ldots, n\} $. Note that $B^c_{1\cup\,2}$ refers to the event where  $\mathcal{R}_1^{\leq k} $ and  $\mathcal{R}_2^{\leq k}$ are disjoint. Since choosing a pair of disjoint subsets of size $k$ from $\{3,\ldots,n\}$ can be done in
    \[\binom{n-2}{k}\binom{n-k-2}{k}\]
    ways, and  $\mathcal{R}_1^{\leq k} $ and  $\mathcal{R}_2^{\leq k}$ given $B^c_{1;2}\cap B^c_{2;1}$ are independent and uniform, we have:
    \begin{align}
        \mathbb{P}(B^c_{1\cup\,2}|B^c_{1;2}\cap B^c_{2;1})& =\frac{\binom{n-2}{k}\binom{n-k-2}{k}}{\binom{n-2}{k}^2}\notag\\
        & = \frac{\binom{n-k-2}{k}}{\binom{n-2}{k}}\notag\\
        & = \frac{(n-k-2)(n-k-3)\cdots(n-2k-1)}{(n-2)(n-3)\cdots(n-k-1)}\notag\\
        & \geq \left(1-\frac{2k-1}{n-2}\right)^k\label{counting}\\
        & = 1 - O(k^2/n) \notag
    \end{align}
    Therefore, 
    \begin{equation*}
        \mathbb{P}(B_{1\cup\,2}|B^c_{1;2}\cap B^c_{2;1})\leq O(k^2/n).
    \end{equation*}
    As a result,
    \begin{align*}
      \mathbb{P}(B_{1\cup\,2}\cap B^c_{1;2}\cap B^c_{2;1})& =\mathbb{P}(B_{1\cup\,2}|B^c_{1;2}\cap B^c_{2;1})\mathbb{P}(B^c_{1;2}\cap B^c_{2;1})  \\
      &\leq O(k^2/n)(1-O(k/n))\\
      &= O(k^2/n).
    \end{align*}
Hence, the following bound holds 
\begin{equation}\label{eq:bound-B12}
    \mathbb{P}(B_{1\cup\,2}\cup B_{1;2}\cup B_{2;1})=\mathbb{P}\big(B_{1\cup\,2}\cap B^c_{1;2}\cap B^c_{2;1}\big)+\mathbb{P}(B_{1;2}\cup B_{2;1})\leq O(k^2/n).
\end{equation}

To bound the probability of $B_{1\rightarrow2}$, we consider $B_{1\rightarrow2}=\big\{\exists i\in\mathcal{R}^{\leq k}_1 \textrm{ such that } (\mathcal{R}^{\leq k}_2 \cup \{2\}) \cap \mathcal{R}^{\leq k}_i \neq \phi\big\}$ given $B_{1\cup\,2}^c\cap B_{1;2}^c\cap B_{2;1}^c$ instead. Let $i\in\mathcal{R}_1^{\leq k}$ be a vertex. An argument similar to \eqref{counting} can be used to show that the probability of $(\mathcal{R}^{\leq k}_2 \cup \{2\}) \cap \mathcal{R}^{\leq k}_i=\phi$ is bounded below by $(1-\tfrac{k+1}{n-1})^k$. As a result, using the union bound, with probability at most $k\big(1-(1-\tfrac{k+1}{n-1})^k\big)$, there exists $i\in\mathcal{R}_1^{\leq k}$ without this property. Hence,
\begin{equation}
        \mathbb{P}(B_{1\rightarrow2}|B^c_{1\cap\,2}\cup B^c_{1;2}\cap B^c_{2;1}) \leq   k\left(1-\left(1-\frac{k+1}{n-1}\right)^k\right)  = O(k^3/n).
        \label{eq:bound-B3}
\end{equation}
Now, we bound $\mathbb{P}(B_{1\rightarrow2})$ using  \eqref{eq:bound-B12} and \eqref{eq:bound-B3}:
\begin{align*}
    \mathbb{P}(B_{1\rightarrow2})  = & \mathbb{P}(B_{1\rightarrow2}|(B_{1\cup\,2}\cup B_{1;2}\cup B_{2;1})^c)\mathbb{P}((B_{1\cup\,2}\cup B_{1;2}\cup B_{2;1})^c) + \\
    & \mathbb{P}(B_{1\rightarrow2}|B_{1\cup\,2}\cup B_{1;2}\cup B_{2;1})\mathbb{P}(B_{1\cup\,2}\cup B_{1;2}\cup B_{2;1}) \\
    \leq & O(k^3/n) + O(k^2/n)\\
    = & O(k^3/n)
\end{align*}
We established the desired bound for $B_{1\rightarrow2}$. The same is true for $B_{2\rightarrow1}$ in a similar way.
\end{proof}

\begin{lem}
 The following two probability bounds hold:
    \begin{align*}
\mathbb{P}(I_1 | B_{1;2}\cup B_{1\cup\,2}\cup B_{1\rightarrow2}) & \leq \exp\left( -k + \frac{1}{2}\sqrt{k}\log(k) + O(\sqrt{k})\right), \text{ and}\\
\mathbb{P}(I_2 | B_{2;1}\cup B_{1\cup\,2}\cup B_{2\rightarrow1}) & \leq \exp\left(-k + \frac{1}{2}\sqrt{k}\log(k)+ O(\sqrt{k})\right).
    \end{align*}
    \label{bound-conditional-isolation}
\end{lem}
\begin{proof}
    We will prove the first inequality, and then the second one will follow by symmetry.  Let $s=\lfloor \sqrt{k}\rfloor \leq k$. We begin by noting that any realization of $LK(n,k)$ conditional on $B_{1;2}\cup B_{1\cup\,2}\cup B_{1\rightarrow2}$ still induces a uniform distribution on the set of all permutations of $\mathcal{E}'_i=\mathcal{E}_i\setminus \{(i,j): j\in\{1,2\}\cup\mathcal{R}_1^s\}$ for $i\neq 2$, and in general on those of
    \[\mathcal{E}' = \bigcup_{i\in \mathcal{R}^{\leq s}_1\setminus\{2\}} \mathcal{E'}_i.\]
    This is because event $B_{1;2}$ does not affect the orders within edges in $\mathcal{E}'$ under the $i\neq 2$ assumption. Moreover, since $(i,2)\notin \mathcal{E}_i'$, $B_{1\cup\,2}$ and $B_{1\rightarrow2}$ only impose order restrictions on edges of form $\{(j,2): j\neq 2\} $, which are disjoint from $\mathcal{E}'$.  Hence, considering permutations of the union $\mathcal{E}_1\cup\mathcal{E}'$ given $B_{1;2}\cup B_{1\cup\,2}\cup B_{1\rightarrow2}$, there are only order restrictions on the elements of $\mathcal{E}_1$. Note that since $\mathcal{E}_1,\mathcal{E}'_{i_1},\mathcal{E}'_{i_2},...$ are disjoint for $i_j\in \mathcal{R}^{\leq s}_1\setminus\{2\}$, elements of $\mathcal{E}_1,\mathcal{E}'_{i_1},\mathcal{E}'_{i_2},...$ can be assigned types $0,1,2,...,s'$, respectively, where $s'=s-1\textrm{ or }s$ depending on whether $2\in\mathcal{R}_1^{\leq s}$ or not. \\
    We now discuss a necessary condition on these permutations for which vertex 1 is isolated. Suppose that vertex 1 is isolated conditional on $B_{1;2}\cup B_{1\cup\,2}\cup B_{1\rightarrow2}$. Therefore, for any $i\in\mathcal{R}_1^{\leq k}$ and specifically for any $i\in\mathcal{R}_1^{\leq s}\setminus\{2\}$, we have $(1,i)\notin \mathcal{R}_i^{\leq k}$. Equivalently, there exist $k$ elements of $\mathcal{E}_i\setminus\{(1,i)\}$ appearing before $(1,i)$ in the permutation. Thus, at least $k-s-2$ elements of $\mathcal{E}'_i$ appear earlier than $(1,i)$. This holds for all $i\in\mathcal{R}_1^{\leq s}\setminus\{2\}$, and since $\mathcal{E}^{\leq s}_1=\{(1,i): i\in\mathcal{R}_1^{\leq s}\}$ are the first $s$ elements of type 0 placed in this permutation, we conclude that there are at most $s$ elements of type 0 in the first $(s-1)(k-s-2)$ elements of this permutation. Lemma \ref{common-lem3} bounds the probability of the event above by 
    \[\mathbb{P}(I_1  | B_{1;2}\cup B_{1\cup\,2}\cup B_{1\rightarrow2}) \leq \exp\left(-k + \frac{1}{2}\sqrt{k}\log(k) + O(\sqrt{k})\right),\]
    which completes the proof.
\end{proof}

\begin{lem}
\label{bound-IB}
    The following holds
    \[\mathbb{P}(I_1I_2\cap B)\leq \exp\left(-\log(n) -k + \frac{1}{2}\sqrt{k}\log(k) + O(\sqrt{k})\right)\]
\end{lem}
\begin{proof} 
We establish the bound through Lemma \ref{bound-B} and \ref{bound-conditional-isolation}.
    \begin{align*}
         \mathbb{P}(I_1I_2\cap B)  & \leq  \mathbb{P}(I_1I_2 \cap (B_{1;2}\cup B_{1\cup\,2}\cup B_{1\rightarrow2}))+ \mathbb{P}(I_1I_2 \cap (B_{2;1}\cup B_{1\cup\,2}\cup B_{2\rightarrow1}))\\
        & \leq  \mathbb{P}(I_1 \cap (B_{1;2}\cup B_{1\cup\,2}\cup B_{1\rightarrow2}))+ \mathbb{P}(I_2 \cap (B_{2;1}\cup B_{1\cup\,2}\cup B_{2\rightarrow1}))\\
        & = \mathbb{P}(I_1 | B_{1;2}\cup B_{1\cup\,2}\cup B_{1\rightarrow2})\mathbb{P}(B_{1;2}\cup B_{1\cup\,2}\cup B_{1\rightarrow2})\\
        &\quad + \mathbb{P}(I_2 | B_{2;1}\cup B_{1\cup\,2}\cup B_{2\rightarrow1})\mathbb{P}(B_{2;1}\cup B_{1\cup\,2}\cup B_{2\rightarrow 1})\\
        & \leq O(k^3/n)\exp\left(-k + \frac{1}{2}\sqrt{k}\log(k) + O(\sqrt{k})\right)\quad (\text{Lemmas }\ref{bound-B}\text{ and }\ref{bound-conditional-isolation})\\
        & = \exp\left(-\log(n) -k + \frac{1}{2}\sqrt{k}\log (k) + O(\sqrt{k})\right).
    \end{align*}
Thus, the result holds.
\end{proof}

\begin{lem}
    We have the following conditional independence
 \[ 
 \mathbb{P}(I_1 I_2 | B^c )  = \mathbb{P}(I_1| B^c ) \mathbb{P}( I_2 | B^c ).
\]
  
    \label{conditional-independence}
\end{lem}

\begin{proof}
Note that $I_1$ holds whenever for each $1\leq i\leq k$ there are at least $k$ elements $j\in[n]\setminus\{1, R_1^i\}$ such that $V(R_1^i, j) > V(R_1^i, 1)$. Assuming $B^c$, this is equivalent to there being at least $k$ elements $j_1\in [n] \setminus \{1,2\} \setminus \{R_2^1,\cdots,R_2^k\}$
with this property. Repeating this argument for $I_2$, for each $1\leq i\leq k$, there must be at least $k$ elements $j_2\in [n] \setminus \{1,2\} \setminus \{R_1^1, R_1^2,\cdots, R_1^k\} $ with $V(R_2^i, j_2) > V(R_2^i, 2)$. Ruling out the possibility of choosing $j_2,j_1$ from $ \{1,2,R_1^1, R_1^2,\cdots, R_1^k\}$, and $\{1,2,R_2^1,\cdots,R_2^k\}$, respectively, implies that there is no common edge among $(R_1^i,1),(R_1^i,j_1),(R_2^i,,2),(R_2^i, j_2)$. Since the scores are chosen independently, and the inequality conditions for the isolation of vertex 1 and that of vertex 2 are disjoint, we obtain conditional independence.
\end{proof}

\begin{lem}\label{lem-second-method-2}
Let $t'<-0.5$ be a real number. Assume $k$ grows to infinity as $n\rightarrow\infty$ with  $k\leq \log(n) + t' \log\log(n)\sqrt{\log(n)}$. Then we have
\[ \limsup_{n\rightarrow\infty} \frac{\mathbb{E}[I_1 I_2]}{\mathbb{E}[I_1]^2} \leq  1.\]
\end{lem}
\begin{proof} 

We begin with the following equality:
\begin{align}
        \frac{\mathbb{P}(I_1)}{\mathbb{P}(B^c)} & = \frac{\mathbb{P}(I_1\cap B^c)}{\mathbb{P}(B^c)} + \frac{\mathbb{P}(I_1 \cap B)}{\mathbb{P}(B^c)}. \label{eq:I1Bcbound}
\end{align}
Hence, it follows that
\begin{align}
        \frac{1}{\mathbb{P}(B^c)} & = \frac{\mathbb{P}(I_1\cap B^c)}{\mathbb{P}(I_1)\mathbb{P}(B^c)} + \frac{\mathbb{P}(I_1 \cap B)}{\mathbb{P}(I_1)\mathbb{P}(B^c)}\notag \\
        & =  \frac{\mathbb{P}(I_1|B^c)}{\mathbb{P}(I_1)} + \frac{\mathbb{P}(I_1 \cap B)}{\mathbb{P}(I_1)\mathbb{P}(B^c)}.
        \label{bbb2}
\end{align}
Lemmas \ref{prob-iso-lower-bound} and \ref{bound-B} imply
\begin{align}
         \frac{\mathbb{P}(I_1 \cap B)}{\mathbb{P}(I_1)\mathbb{P}(B^c)} & \leq O\left(\frac{k^3}{n}\right)e^{k\left(1+\sqrt{\frac{6\log(k)}{k}}\right)+o(1)}\notag\\
         & =\exp\left(k+\sqrt{6k\log(k)} + 3\log(k) - \log(n)+O(1)\right)\notag\\
         & \leq \exp\left( t' \log\log(n)\sqrt{\log(n)} + \sqrt{6k\log(k)} + 3\log(k))) + O(1)\right) \notag \\
         &  \leq \exp\left( t' \log(k)\sqrt{k} + \sqrt{6k\log(k)} + 3\log(k)+O(1)\right) \\
          & \rightarrow_{k\rightarrow\infty} 0, \notag
\end{align}
where we use the fact that $k \leq \log(n)$. Considering the fact $\tfrac{1}{\mathbb{P}(B^c)}\rightarrow 1$, equation \eqref{bbb2} implies that
\begin{equation}
    \frac{\mathbb{P}(I_1|B^c)}{\mathbb{P}(I_1)}\rightarrow 1.
    \label{single-asym}
\end{equation}
It follows from Lemmas \ref{bound-IB}, \ref{conditional-independence} that
\begin{align}
    \mathbb{P}(I_1 I_2)  = & \mathbb{P}(I_1 I_2 \cap B^c ) + \mathbb{P}(I_1 I_2 \cap B )\notag\\
    \leq & \mathbb{P}(I_1 I_2 \cap B^c ) + \exp\left(-\log(n) -k + \frac{1}{2}\sqrt{k}\log(k) + O(\sqrt{k})\right)\notag \\
    \leq & \mathbb{P}(I_1 I_2 | B^c ) + \exp\left(-\log(n) -k + \frac{1}{2}\sqrt{k}\log(k) + O(\sqrt{k})\right)\notag \\
     = & \mathbb{P}(I_1| B^c ) \mathbb{P}( I_2 | B^c ) + \exp\left(-\log(n) -k + \frac{1}{2}\sqrt{k}\log(k) + O(\sqrt{k})\right).
    \label{bbb3}
\end{align}
Dividing both sides of (\ref{bbb3}) by $\mathbb{P}(I_1)\mathbb{P}(I_2)=\mathbb{P}(I_1)^2$ and using Lemma \ref{prob-iso-lower-bound}, we have
\begin{align}
    & \frac{\mathbb{P}(I_1 I_2)}{\mathbb{P}(I_1)\mathbb{P}(I_2)} \notag \\& \leq \frac{\mathbb{P}(I_1| B^c ) \mathbb{P}( I_2 | B^c )}{\mathbb{P}(I_1)\mathbb{P}(I_2)} + \frac{\exp\left(-\log(n) -k + \frac{1}{2}\sqrt{k}\log(k) + O(\sqrt{k})\right)}{{\mathbb{P}(I_1)\mathbb{P}(I_2)}}\notag \\
    & \leq \frac{\mathbb{P}(I_1| B^c ) \mathbb{P}(I_2 | B^c )}{\mathbb{P}(I_1)\mathbb{P}(I_2)} + \frac{\exp\left(-\log(n) -k +\frac{1}{2}\sqrt{k}\log(k) + O(\sqrt{k})\right)}{\exp\left( -2k\left(1+\sqrt{\frac{6\log(k)}{k}}\right)+o(1)\right)}\notag \\
    & \leq \frac{\mathbb{P}(I_1| B^c ) \mathbb{P}( I_2 | B^c )}{\mathbb{P}(I_1)\mathbb{P}(I_2)} + \exp\left(-\log(n) + k +\frac{1}{2}\sqrt{k}\log(k)+\sqrt{24k\log(k)} + O(\sqrt{k}) \right) \notag \\
    & \leq \frac{\mathbb{P}(I_1| B^c ) \mathbb{P}( I_2 | B^c )}{\mathbb{P}(I_1)\mathbb{P}(I_2)} \notag \\ 
    & \quad + \exp\left(t'\log\log(n)\sqrt{\log(n)}+\frac{1}{2}\sqrt{k}\log(k) +\sqrt{24k\log(k)}+ O(\sqrt{k}) \right) \notag \\
    & \leq \frac{\mathbb{P}(I_1| B^c ) \mathbb{P}( I_2 | B^c )}{\mathbb{P}(I_1)\mathbb{P}(I_2)} +  \exp\left((0.5+t')\log(k)\sqrt{k} +\sqrt{24k\log(k)}+ O(\sqrt{k}) \right).
    \label{final-step}
\end{align}
Equation \eqref{single-asym} guarantees that the first term in right hand side converges to 1. The second term vanishes as $k\rightarrow\infty$ as a result of $t'<-0.5$. Since $\mathbb{E}[I_1]=\mathbb{P}(I_1)$ and $\mathbb{E}[I_1I_2]=\mathbb{P}(I_1I_2)$, this completes the proof.
\end{proof}

\begin{thm}\label{thm:disconnected}
The random graph model $\mathit{LK}(n,k)$ with $k \leq \log (n) + t'\log \log (n)\sqrt{\log (n)}$ for $t'<-\tfrac{1}{2}$ is not connected with high probability. In particular, if $k=\lfloor t\log (n) \rfloor$ for $0<t<1$, the graph is disconnected with high probability.
\end{thm}
\begin{proof}
Having both Lemmas \ref{lem-second-method-1} and \ref{lem-second-method-2} one can apply the second moment method. We begin with
\begin{align}
          \mathbb{E}\left[\left(\sum_{i=1}^n I_i \right)^2\right] - \mathbb{E}\left[\sum_{i=1}^n I_i \right]^2 
         &  = \mathrm{Var}\left[\sum_{i=1}^n I_i \right] \notag\\ 
         & \geq \left(0 - \mathbb{E}\left[\sum_{i=1}^n I_i \right] \right)^2 \mathbb{P}\left(\sum_{i=1}^n I_i = 0 \right). 
\end{align}
Therefore,
\begin{align}
    \mathbb{P}\left(\sum_{i=1}^n I_i = 0\right) & \leq \frac{\mathbb{E}[(\sum_{i=1}^n I_i )^2] - \mathbb{E}[\sum_{i=1}^n I_i ]^2}{\mathbb{E}[\sum_{i=1}^n I_i ] ^2 }  \notag\\
    & \leq \frac{\mathbb{E}[(\sum_{i=1}^n I_i )^2]}{\mathbb{E}[\sum_{i=1}^n I_i ] ^2 } - 1 \notag\\
     & = \frac{\sum_{i=1}^n \mathbb{E}[I_i^2]+\sum_{i,j,i\neq j}\mathbb{E}[I_i I_j]}{n^2 \mathbb{E}[I_1]^2} - 1 \notag\\
     & = \frac{n \mathbb{E}[I_1] + n(n-1) \mathbb{E}[I_1 I_2]}{n^2 \mathbb{E}[I_1]^2} - 1 \notag\\
     & = \frac{1}{n \mathbb{E}[I_1]}+\frac{(n-1) \mathbb{E}[I_1 I_2]}{n \mathbb{E}[I_1]^2} - 1.
    \label{ineq1}
\end{align}
Now it follows from Lemma \ref{lem-second-method-1} that the first fraction vanishes as $n\rightarrow\infty$. That is, for any $\epsilon>0$, there exists $N_1>0$ such that $ \frac{1}{n \mathbb{E}[I_1]}\leq\frac{\epsilon}{2}$ for all $n>N_1$. Moreover, Lemma \ref{lem-second-method-2} implies that the second fraction will be sufficiently close to 1. Indeed, there exists $N_2>0$, such that for all $n>N_2$, we have $\frac{ \mathbb{E}[I_1 I_2]}{ \mathbb{E}[I_1]^2}\leq 1 + \frac{\epsilon}{2}$. Hence, for any given $\epsilon>0$, we find $N=\max(N_1,N_2)$ such that for any $n>N$:
\begin{align*}
    0& \leq\mathbb{P}\left(\sum_{i=1}^n I_i = 0\right)\\
    & \leq \frac{1}{n \mathbb{E}[I_1]}+\frac{(n-1) \mathbb{E}[I_1 I_2]}{n \mathbb{E}[I_1]^2} - 1\\
    & \leq \frac{\epsilon}{2}+ 1 + \frac{(n-1)\epsilon}{2n} - 1\\
    & \leq \epsilon
\end{align*}
This shows $\mathbb{P}\left(\sum_{i=1}^n I_i = 0\right)\rightarrow0$ as $n\rightarrow\infty$. Therefore, the probability of having no isolated vertex converges to $0$, ensuring the existence of an isolated vertex with high probability as $n$ grows to infinity. 
\end{proof}

\section{Connectivity of $\mathit{LK}(n,k)$ Graphs for $t>1$}\label{sec:connectivity}

A common method to prove connectivity of random graphs is by locating (with high-probability) a Erd\H{o}s-R\'enyi sub-graph consisting of all the vertices. Then, the celebrated result of Erd\H{o}s and R\'enyi \cite{erdHos1960evolution,erdHos1959randomgraph,renyi1959connected} guarantees connectivity if the probability of an edge being present is $p = \tfrac{t\log(n)}{n}$ where $t>1$. This method used in \cite{la2015new} proves the connectivity of $LK(n,k)$ for $k=t\log(n)$ where $t>C=2.4625$. The link between Erd\H{o}s-R\'enyi and $LK(n,k)$ is made through a concentration property of order statistics. The authors of \cite{la2015new} assume that the edges are independently scored by $\mathrm{Exp}(1)$ random variables which results in the presence of edges whose scores are greater than or equal $\log((n-1)/(t\log(n))) + \sqrt{2/t}$ with high probability. Therefore, if $L_{i,j}$ is a random variable distributed as $\mathrm{Exp}(1)$, then
\[\mathbb{P}\left(L_{i,j}>\log\left(\frac{n-1}{t\log(n)}\right) + \sqrt{\frac{2}{t}}\right)= \frac{t\log(n)}{n-1} e^{-\sqrt{\frac{2}{t}}}.\]
It can be easily shown that $te^{-\sqrt{\tfrac{2}{t}}}>1$ for $t>2.4625$. This provides an independent possibility of connection for all edges with the probability of $t'\log (n)/ (n-1)$ for a $t'=te^{-\sqrt{\frac{2}{t}}}>1$. Thereafter, the connectivity result of Erd\H{o}s-R\'enyi random graphs~\cite{erdHos1959randomgraph,erdHos1960evolution,renyi1959connected} implies the connectivity of $\mathit{LK}(n,k)$.

We prove the connectivity of $\mathit{LK}(n,k)$ for all $t>1$ in two steps. First, we rule out the possibility of having components of size $O(1)$. Second, we apply the idea from \cite{la2015new} described above to find an Erd\H{o}s-R\'enyi graph with all edges contained in $\mathit{LK}(n,k)$. As opposed to \cite{la2015new}, we do not restrict $t$ to find a $t'>1$.  Hence, the resulting Erd\H{o}s-R\'enyi graph is not necessarily connected. However, by modifying the original proof of the connectivity of Erd\H{o}s-R\'enyi graphs for $t'>1$ from \cite{renyi1959connected}, we can prove a weaker result for an arbitrary $t'>1$ which then helps us deduce the connectivity result. 

\begin{thm}
Let $\kappa\geq 0$ be a non-negative integer. Consider random graphs of class $\mathit{LK}(n,k)$ where $k \geq \log(n) +  t'\log\log(n)\sqrt{\log(n)} $ for a real number $t'>\tfrac{1}{2}$ . Then
\[\mathbb{P}\big(\exists \textrm{ a vertex of degree less than or equal } \kappa\big)\rightarrow 0\]
as $n\rightarrow\infty$.
\label{exclude-o(1)}
\end{thm}
\begin{proof}
We will bound the probability that vertex $1$ has degree less than $\kappa$ with the same technique used in the proof of Lemma \ref{bound-conditional-isolation}. Let $s=\lfloor \sqrt{k}\rfloor \leq k$ and for $i\in \mathcal{R}^{\leq s}_1$, $\mathcal{E'}_{i} = \mathcal{E}_{i} \setminus \{(i,j):j\in\mathcal{R}^{\leq s}_1\cup\{1\}\}$. Now any realization of $LK(n,k)$ induces a uniform distribution on the set of all permutations of the following union
\[\mathcal{E}' = \mathcal{E}_1 \bigcup_{i\in \mathcal{R}^{\leq s}_1} \mathcal{E'}_i, \]
with some order restrictions only within the elements of $\mathcal{E}_1$. The union above is a partition in which $|\mathcal{E}_1|=n-1$, and $|\mathcal{E'}_i|=n-s-1$ for all $i\in \mathcal{R}^{\leq s}_1$. We assign type 0 to the elements of $\mathcal{E}_1$ and type $i$ to the elements of $\mathcal{E'}_i$ for $i\in\mathcal{R}_1^{\leq s}$. Therefore, any realization of $LK(n,k)$ induces a uniformly random permutation over $\mathcal{E}'$ with some order restrictions on the elements of type 0.

A necessary condition to ensure that the degree of vertex 1 is bounded above by $\kappa$ is that for at least $s-\kappa$ elements $i\in \mathcal{R}^{\leq s}_1$, there exist at least $k$ edges of $\mathcal{E}_{i}$ appearing before $(1,i)$ in the permutation. As a result, those $k$ edges appear before the $s^{th}$ element of $\mathcal{E}_{1}$ too. This implies there must be at least $k-s-1$ edges of $\mathcal{E'}_{i}$ before the $s^{th}$ element of $\mathcal{E}_{1}$. As a result, in the first $(s-\kappa)(k-s-1)$ elements of this permutation, there are at most $s-1$ elements of type 0. Lemma \ref{common-lem3} bounds the probability of the event above by 

\[\exp\left(-k + \frac{1}{2}\sqrt{k}\log(k) + O(\sqrt{k})\right)\]
Then, using the union bound for sufficiently large $n$ we have:
\begin{align*}
        \mathbb{P}\big(\exists v:\deg(v)\leq  \kappa\big) & \leq n\times \mathbb{P}\big(\textrm{vertex $1$ is of degree at most } \kappa\big) \\
        & \leq n \exp\left(-k + \frac{1}{2}\sqrt{k}\log(k) + O(\sqrt{k})\right) \\
        &=\exp\left(\log(n) - k + \frac{1}{2}\sqrt{k}\log(k) + O(\sqrt{k})\right)
\end{align*}
If $k>O(\log(n))$, then the right-hand side obviously vanishes. In case of $k=O(\log(n))$, we use the fact that $ k - \frac{1}{2}\sqrt{k}\log(k)+O(\sqrt{k})$ is increasing for sufficiently large $k$, and replace $k$ by $\log(n)+  t'\log\log(n)\sqrt{\log(n)}$ in the last equality to obtain
\begin{align*}
        \exp\left( \big(0.5-t'\big)\log\log(n)\sqrt{\log(n)}+O\left(\sqrt{\log(n)}\right)\right) \rightarrow  0.
\end{align*}
Hence, the convergence holds.
\end{proof}
Theorem \ref{exclude-o(1)} excludes the possibility of having components of size $O(1)$. In other words, the following corollary holds.
\begin{cor}
Consider the random graphs model of class $\mathit{LK}(n,k)$ with $k \geq \log(n) +  t'\log\log(n)\sqrt{\log(n)} $ for a real number $t'>\tfrac{1}{2}$.  Then
\[\mathbb{P}\big(\exists \textrm{ a component of size at most } \kappa\big)\rightarrow 0\]
for any fixed $\kappa$ as $n\rightarrow\infty$. In particular when $k=\lfloor t\log(n) \rfloor$ for $t>1$ the limit above is still true.
\label{finite-comp}
\end{cor}
\begin{proof}
In case of having a component of size $\kappa$ there must be a vertex of degree at most $\kappa$, which is impossible with high probability due to Theorem \ref{exclude-o(1)}.
\end{proof}

Next, we establish a proof based on finding an Erd\H{o}s-R\'enyi sub-graph and super-graph of $\mathit{LK}(n,k)$ with both containing all the vertices. Recall that the order statistics of a set of \emph{i.i.d.} random variables $\{X_1,X_2,\cdots,X_n\}$ is defined as their non-increasing rearrangement $X^{(1)}\geq X^{(2)}\geq \cdots \geq X^{(n)}$. Recall that the $k^{th}$ order statistic of $\mathcal{V}_i$ is denoted by $V(i,R_i^k)$.
\begin{thm}
For the random graph model $\mathit{LK}(n,k)$ with $k=\lfloor t \log(n) \rfloor $, there is no component of size $\lceil \tfrac{8.24
}{t}\rceil \leq r \leq \lfloor \tfrac{n}{2}\rfloor$ with high probability.
\label{ER-estim}
\end{thm}
\begin{proof}
First, we refer to Lemma A.1 from \cite{la2015new}, which states that if the scores \( V(i,j) \) are independently distributed as exponential random variables with parameter \(1\), and we define
\[A_n = \left\{\forall \; 1\leq i\leq n: \;V(i,R_i^k)\in \left(\log\left(\frac{n-1}{t\log(n)}\right)-\sqrt{2} ,  \log\left(\frac{n-1}{t\log(n)}\right)+\sqrt{2}   \right)  \right\},\]
then $\mathbb{P}(A_n)\rightarrow 1$ as $n\rightarrow\infty$. We let $\underline{l}=\log(\frac{n-1}{t\log(n)})-\sqrt{2}$ and $\bar{l} = \log(\frac{n-1}{t\log(n)})+\sqrt{2}$. Next, note that 
\begin{align*}
\bar{p} &:= \mathbb{P}\big(V(i,j)> \bar{l}\big) = \frac{t\log(n)}{n-1}e^{-\sqrt{2}}\approx\frac{0.2431t\log(n)}{n-1},\\
\underline{p} &:= \mathbb{P}\big(V(i,j)> \underline{l}\big)= \frac{t\log(n)}{n-1}e^{\sqrt{2}}\approx \frac{4.1133t\log(n)}{n-1}.
\end{align*}
Additionally, if the graph meets condition \(A_n\) (which holds \textit{with high probability}), then \( V(i,j) > \bar{l} \) indicates that the edge \((i,j)\) exists, while \( V(i,j) < \underline{l} \) ensures that the edge \((i,j)\) does not exist.

Let the graph satisfy condition $A_n$. In order to have a component of size $r$, one must choose $r$ vertices containing at least one spanning tree. Moreover, according to Cayley's formula, there are $r^{r-2}$ possible trees with $r$ vertices. Thus, the $r-1$ edges of the tree must have scores no less than $\underline{l}$ which occurs with probability $\underline{p}^{r-1}$ due to independence. Additionally, we require that those $r$ vertices are not connected to the rest of the graph, implying that the scores for the intermediate edges between the component and the rest of the graph must be less than $\bar{l}$. Note that both of these are necessary conditions. Counting the number of possible components yields the following upper bound:
\begin{align*}
\Pi & =\mathbb{P}\left(\exists \textrm{ a component of size  between } \lceil 8.24/t \rceil \textrm{ and } \lfloor n/2 \rfloor \right) \\
& \leq\sum_{r=\lceil 8.24/t\rceil}^{\lfloor n/2\rfloor}\binom{n}{r}r^{r-2}\underline{p}^{r-1}(1-\bar{p})^{r(n-r)}.
\end{align*}
Substituting $\binom{n}{r}<\tfrac{n^r}{r!}<n^r/\sqrt{2\pi r}(r/e)^r$ implies
\begin{align*}
        \Pi & <  \frac{1}{\sqrt{2\pi}} \sum_{r=\lceil \frac{8.24
}{t}\rceil}^{\lfloor n/2\rfloor}n^r r^{-r-1/2} e^r r^{r-2}\underline{p}^{r-1}(1-\bar{p})^{r(n-r)}\\
         & <  \frac{1}{\sqrt{2\pi}} \sum_{r=\lceil \frac{8.24
}{t}\rceil}^{\lfloor n/2\rfloor}e^r \frac{n^r}{\underline{p}} r^{-5/2}\underline{p}^{r}e^{-r(n-r)\bar{p}}\quad (\textrm{using}\;1-x<e^{-x})\\
         & <  \frac{n}{\sqrt{2\pi}} \sum_{r=\lceil \frac{8.24
}{t}\rceil}^{\lfloor n/2\rfloor}n^r e^r \underline{p}^{r}e^{-r\bar{p}n/2}\quad (\textrm{using}\;1/\underline{p}<n, \; n-r>n/2, \; r^{-5/2}<1)\\
         & < \frac{n}{\sqrt{2\pi}} \sum_{r=\lceil \frac{8.24
}{t}\rceil}^{\lfloor n/2\rfloor} e^{r(1 + \log(n\underline{p}) - n\bar{p}/2)}.
\end{align*}
As $t$ is a fixed number, the dominant term in $1 + \log(n\underline{p}) - n\bar{p}/2$ is $- n\bar{p}/2\approx -0.1216 t\log(n)$. As a result, for sufficiently large $n$, $1 + \log(n\underline{p}) - n\bar{p}/2 < -0.1215t\log(n)$. Thus, we have
\begin{align*}
        \Pi & < \frac{n}{\sqrt{2\pi}} \sum_{r=\lceil \frac{8.24
}{t}\rceil}^{\infty} e^{-0.1215tr\log(n)}\\
         & =  \frac{n}{\sqrt{2\pi}} \sum_{r=\lceil \frac{8.24
}{t}\rceil}^{\infty} n^{-0.1215tr}\\
         & = \frac{n^{1 - 0.1215t\lceil \frac{8.24
}{t}\rceil}}{\sqrt{2\pi}(1-n^{- 0.1215t})} \\
 & \leq  \frac{n^{- 0.00116}}{\sqrt{2\pi}(1-n^{- 0.1215t})}\rightarrow_{n\rightarrow\infty} 0.
\end{align*}
Therefore, \textit{with high probability} there is no component of size between $\lceil 8.24/t\rceil$ and $n/2$. 
\end{proof}
A key feature of Theorem \ref{ER-estim} is its validity for any positive $t$. This means that if $k\approx t\log(n)=\Theta(\log(n))$, one should only examine components of size less than $\lceil \frac{8.24
}{t}\rceil=O(1)$ to study the disconnectedness of the graph. Since Corollary \ref{finite-comp} addresses components of size $O(1)$, we now assert the main theorem.
\begin{thm}\label{thm:connected}
The random graph model $\mathit{LK}(n,k)$ with $k \geq \log( n )+  t'\log\log (n)\sqrt{\log(n)} $ for a real number $t'>\tfrac{1}{2}$ is connected with high probability. In particular, if $k=\lfloor t\log(n)\rfloor$ for $t>1$, connectivity holds with high probability. 
\end{thm}
\begin{proof}
If a graph from class $\mathit{LK}(n,k)$ is disconnected, then it should have at least a component of size $s$ where $1\leq s \leq \lfloor n/2 \rfloor $. Since $\lceil \tfrac{8.24}{1} \rceil=9$, Theorem \ref{ER-estim} rules out the possibility of having components of size between 9 and $\lfloor n/2 \rfloor$ with high probability. Moreover, Corollary \ref{finite-comp} guarantees that the probability of having components of size at most $8$ vanishes, when $n$ grows to infinity. Therefore, the graph will be connected \textit{with high probability}. 
\end{proof}

\section{Average Degree of $\mathit{LK}(n,k)$ Graphs}\label{sec:avgdegree}

Here, we discuss another set of results for $\mathit{LK}(n,k)$ random graphs. From the construction it is immediate that the degree sequence is bounded above by $k$. Therefore, the average degree is also constrained by $k$. However, we will show that for a large range of $k$ (as a function of $n$), this number is very close to $k$ by specifying the error term. Finally, we compare our result to the results on the sparse case in \cite{moharrami2024erlang}. 

Before presenting the main theorem we recall the negative binomial distribution. Imagine an urn with an infinite number of red and blue balls, where each draw results in a red ball with probability $p$. The probability of drawing $j$ blue balls before the $k^{th}$ red ball is a negative binomial distribution given by the following formula $\mathbb{P}(X_k=j)=\binom{k+j-1}{j}p^k(1-p)^{j}$. We need to introduce two combinatorial Lemmas before stating the main Theorem. 
\begin{lem}
    \[\sum_{j=0}^{k-1} \binom{k+j-1}{j} \frac{1}{2^{k+j}}=\frac{1}{2}\]
    \label{negative-binom-mean1}
\end{lem}
\begin{proof}
     Let $X_k$ denote the negative binomial random variable with $p=\frac{1}{2}$. Note that $\mathbb{P}(X_k<k)=\sum_{j=0}^{k-1} \binom{k+j-1}{j} \frac{1}{2^{k+j}}$ is the probability of observing the $k^{th}$ red ball earlier than the $k^{th}$ blue ball in the infinite urn model. By symmetry this value must be $1/2$.
\end{proof}
\begin{lem}
    \[\sum_{j=0}^{k-1}\binom{k+j-1}{j}\left(\frac{j}{2^{k+j}}\right)=\frac{k}{2}-\frac{(2k-1)}{2^{2k-1}}\binom{2k-2}{k-1}\]
    \label{negative-binom-mean2}
\end{lem}

\begin{proof}
    Using Lemma \ref{negative-binom-mean1}, one can write the sum above as follows 
    \begin{align*}
        A & = \sum_{j=0}^{k-1}\binom{k+j-1}{j}\left(\frac{j}{2^{k+j}}\right)\\
        & = \sum_{j=1}^{k-1}\binom{k+j-2}{j-1}\left(\frac{k+j-1}{2^{k+j}}\right)\\
        & = \left(\frac{k}{2}\right)\sum_{j=1}^{k-1}\binom{k+j-2}{j-1}\frac{1}{2^{j+k-1}} + \left(\frac{1}{2}\right)\sum_{j=1}^{k-1}(j-1)\binom{k+j-2}{j-1}\frac{1}{2^{j+k-1}}\\
        & = \left(\frac{k}{2}\right)\left[\frac{1}{2}-\binom{2k-2}{k-1}\frac{1}{2^{2k-1}}\right] + \left(\frac{1}{2}\right)\left[A - (k-1)\binom{2k-2}{k-1}\frac{1}{2^{2k-1}} \right]\\
        & = \frac{k}{4}+\frac{A}{2} - \frac{(2k-1)}{2^{2k-2}}\binom{2k-2}{k-1}.
    \end{align*}
    We now solve the resulting equation for $A$ to get the desired expression.
\end{proof}

\begin{thm}\label{thm:genasymptotic}
In the model $\mathit{LK}(n,k)$ suppose $D$ represents the degree of a randomly chosen vertex. If $k=o(\sqrt{n})$ then we have the following asymptotic
\begin{equation}
    \mathbb{E}[D]=k - \left[ \frac{(2k-1)}{2^{2k-1}}\binom{2k-2}{k-1} \right](O(k^2/n)+1)
    \label{original-ave-deg-eq}
\end{equation}
\label{original-ave-deg}
\end{thm}
\begin{proof}
Without loss of generality assume the randomly chosen vertex is vertex 1 and $R_1^i = i+1$ for all $1\leq i\leq n-1$. Let $E_i$ denote the indicator function for whether vertex 1 is connected to vertex $i+1$. To prevent the graph from having edge $(1,i+1)$, vertex $i+1$ must disagree on vertex 1. This is equivalent to appearing at least $k$ elements of $\mathcal{E}_{i+1}$ before the $i^{th}$ element of $\mathcal{E}_1$, which is $(1,i+1)$, in the permutation. Hence, there could be at most $i$ elements of $\mathcal{E}_1$ before the $k^{th}$ element of $\mathcal{E}_{i+1}$. Any realization of $LK(n,k)$ induces a uniformly random permutation on $\mathcal{E}_1\cup \mathcal{E}_{i+1}$ subject to the order restriction $(1,2)\succ (1,3)\succ\cdots\succ (1,n)$. By employing a similar approach as in the proof of Lemma \ref{common-lem2}, we can use combinatorial methods to calculate the probability of observing $j$ elements of $\mathcal{E}_1$ earlier than the $k^{th}$ element of $\mathcal{E}_{i+1}$ for $j\leq k$ as follows: 
\begin{align*}
    \binom{k+j-1}{j}\frac{(n-1)(n-2)\cdots(n-k-j)}{(2n-3)(2n-4)\cdots(2n-k-j-2)} & = \binom{k+j-1}{j} \left(\frac{n+O(k)}{2n+O(k)}\right)^{k+j}\\ & = \binom{k+j-1}{j} \left(\frac{1}{2}\right)^{k+j}(1 + O(k^2/n))
\end{align*}

Varying $0\leq j \leq i-1$, it follows for each $1\leq i \leq k$ that:
\[\mathbb{P}( E_i = 0 ) = \left[\sum_{j=0}^{i-1} \binom{k+j-1}{j} \frac{1}{2^{k+j}}\right](1+O(k^2/n)).\]
 We now compute $\mathbb{E}[D]$ using Lemmas \ref{negative-binom-mean1} and \ref{negative-binom-mean2} as follows
\begin{align*}
        \mathbb{E}[D]  & = \sum_{i=1}^{n-1} \mathbb{E}[E_i] \\
        & = \sum_{i=1}^{n-1} \mathbb{P}(E_i = 1 ) \\
        & = \sum_{i=1}^{k} \mathbb{P}(E_i = 1 ) \\
        & = k - \sum_{i=1}^{k} \mathbb{P}(E_i = 0 )\\
        & = k - \sum_{i=1}^{k} \left[\sum_{j=0}^{i-1} \binom{k+j-1}{j} \frac{1}{2^{k+j}}\right](O(k^2/n)+1) \\
        & = k - \left[\sum_{j=0}^{k-1} \binom{k+j-1}{j} \frac{k-j}{2^{k+j}}\right](O(k^2/n)+1)\\
        & = k - \left[\frac{k}{2}-\sum_{j=0}^{k-1} \binom{k+j-1}{j} \frac{j}{2^{k+j}}\right](O(k^2/n)+1)\\
        & = k - \left[ \frac{(2k-1)}{2^{2k-1}}\binom{2k-2}{k-1} \right](O(k^2/n)+1). 
\end{align*}
This completes the proof.
\end{proof}

\smallskip
\noindent\textbf{Remark:} As mentioned earlier, the work  
\cite{moharrami2024erlang} by Moharrami \textit{et al.} considers the preference threshold to be a random variable independently chosen per vertex using a distribution $P$ over $\mathbb{N}$ with finite mean instead of just a fixed $k$ for the potential number of neighbors of a randomly chosen vertex. Further, 
\cite[Theorem 5.1]{moharrami2024erlang} 
specifies the following formula for the average degree in the limit of $n$ going to infinity 
\begin{equation}
    \mathbb{E}[D] = \sum_{i=1}^{\infty}\sum_{j=1}^{\infty} P(i)P(j)\int_0^{\infty}\bar{F}_i(x)\bar{F}_j(x)dx 
    \label{ave-deg-mehrdad}
\end{equation}
where $\bar{F}_i$ denotes the complementary 
cumulative distribution function of $\textrm{Erlang}(\cdot\, ;i,1)$ (the Erlang distribution of shape $i$ and rate $1$).

In the case of bilateral preference graphs with preference threshold parameter to be a fixed $k$, the probability distribution $P$ becomes a delta mass function on $k$, i.e., $P(i)=1$ if $i=k$, and $P(i)=0$ otherwise. In addition, the mean being finite as $n$ goes to infinity implies that $k$ must be finite and cannot grow to infinity. The sum in \eqref{ave-deg-mehrdad} now simplifies to
\begin{equation}
    \mathbb{E}[D] = \int_0^{\infty}\bar{F}_k^2(x)dx.
\end{equation}
Then using $\bar{F}_k(x)=\sum_{i=0}^{k-1}\frac{1}{i!}e^{-x}x^i$, distributing the square, and exchanging the finite sum with integral, we obtain
\begin{align}
    \mathbb{E}[D] & = \sum_{0\leq i,j\leq k-1} \int_0^{\infty} \frac{x^{i+j}}{i!j!}e^{-2x}dx\notag\\
    & = \sum_{0\leq i,j\leq k-1} \int_0^{\infty} \frac{(2x)^{i+j}}{2^{i+j}i!j!}e^{-2x} dx\notag\\
    & = \sum_{0\leq i,j\leq k-1} \frac{\Gamma(i+j+1)}{2^{i+j+1} i! j!}\notag\\
    & = \frac{1}{2}\sum_{0\leq i,j\leq k-1} \frac{1}{2^{i+j}} \binom{i+j}{i}.
\end{align}
We substitute $s=i+j$ and denote the negative binomial random variable with parameter $p=1/2$ for representing the probability of appearing $k^{th}$ red ball after observing $X_k$ blue balls by $X_k$. This leads us to the following expression:
\begin{align}
        \mathbb{E}[D] & = \frac{1}{2} \left[ \sum_{s=0}^{k-1} \sum_{i=0}^s \frac{1}{2^{s}} \binom{s}{i} + \sum_{s=k}^{2k-2}\sum_{i=s-k+1}^{k-1} \frac{1}{2^s}\binom{s}{i}\right]\notag\\
        & = \frac{1}{2} \left[ k + \sum_{s=0}^{k-2}\left(1-2\sum_{i=0}^{s}\frac{1}{2^{s+k}}\binom{s+k}{i}\right)\right]\notag\\
        & = \frac{1}{2} \left[ k + \sum_{s=0}^{k-2}\left(1-2\mathbb{P}(X_k\leq s )\right)\right]\notag\\
        & = \frac{1}{2} \left[ k + \sum_{s=0}^{k-2}\left(2\mathbb{P}(X_k> s )-1\right)\right]\notag\\
        & = \frac{1}{2} \left[ 1 + 2\mathbb{E}[X_k 1_{X_k\leq k-1}]+2(k-1)\mathbb{P}(X_k\geq k)\right]\notag\\
        & = \frac{1}{2} \left[ 1 + \sum_{j=0}^{k-1}\binom{k+j-1}{j}\frac{2j}{2^{k+j}} +2(k-1)(1-\mathbb{P}(X_k< k))\right]\notag\\
        & = \frac{k}{2} + \sum_{j=0}^{k-1}\binom{k+j-1}{j}\frac{j}{2^{k+j}} \quad (\text{Using Lemma }\ref{negative-binom-mean1})\notag\\
        & = k -  \left[ \frac{(2k-1)}{2^{2k-1}}\binom{2k-2}{k-1} \right]. \quad (\text{Using Lemma }\ref{negative-binom-mean2})
    \label{ave-deg-mehrdad-to-mine}
\end{align}
Therefore, we reproduce the same formula for the average degree as in Theorem \ref{original-ave-deg} using the result in 
\cite{moharrami2024erlang}. 
We discussed this remark with the assumption of sparseness. However, Theorem \ref{original-ave-deg} only needs $k=o(\sqrt{n})$ which covers both the sparse and non-sparse regimes.

Next we determine the asymptotic behavior of the mean degree as both $n$ and $k$ go to infinity with a constraint on how fast $k$ grows relative to $n$.

\begin{thm}\label{thm:meandegree}
In the model $\mathit{LK}(n,k)$, assume that $k=o(n^{1/3})$ grows to infinity as $n$ goes to infinity. Suppose $D$ represents the degree of a randomly chosen vertex. We have the following asymptotic behavior of the average degree
\[\mathbb{E}[D]= k - \sqrt{\frac{k}{\pi}} + \frac{1}{8\sqrt{\pi k}} + o\left(\frac{1}{\sqrt{k}}\right).\]
\end{thm}
\begin{proof}
    We simply use Stirling's approximation $n!=\sqrt{2\pi n} (\tfrac{n}{e})^n \left(1+\tfrac{1}{12n}+o(\tfrac{1}{n})\right)$ to find an asymptotic expression for $\binom{2k-2}{k-1}$, namely,
    \begin{align*}
      \binom{2k-2}{k-1}  & = \frac{(2k-2)!}{(k-1)!^2}  \\
      & = \frac{\sqrt{2\pi (2k-2)} (\frac{2k-2}{e})^{2k-2} \left(1+\frac{1}{12(2k-2)}+o(\frac{1}{k})\right)}{2\pi (k-1)  (\frac{k-1}{e})^{2k-2} \left(1+\frac{1}{12(k-1)}+o(\frac{1}{k})\right)^2} \\
      & = \frac{\sqrt{2\pi (2k-2)} 2^{2k-2} \left(1+\frac{1}{12(2k-2)}+o(\frac{1}{k})\right)}{2\pi (k-1)  \left(1+\frac{1}{12(k-1)}+o(\frac{1}{k})\right)^2} \\
      & = \frac{2^{2k-2} \left(1+\frac{1}{24(k-1)}+o(\frac{1}{k})\right)}{\sqrt{\pi(k-1)} \left(1+\frac{1}{6(k-1)}+o(\frac{1}{k})\right)} \\
      & = \left(\frac{2^{2k-2}}{\sqrt{\pi(k-1)}}\right)\left(1 - \frac{1}{8(k-1)} + o\left(\frac{1}{k}\right)\right).
    \end{align*}
  Substituting this into \eqref{original-ave-deg-eq} implies
    \begin{align*}
        \mathbb{E}[D] & =k - \left[ \frac{(2k-1)}{2^{2k-1}}\binom{2k-2}{k-1} \right](O(k^2/n)+1)\\
        & = k - \left[ \frac{(2k-1)}{2^{2k-1}} \left(\frac{2^{2k-2}}{\sqrt{\pi(k-1)}}\right)\left(1 - \frac{1}{8(k-1)} + o\left(\frac{1}{k}\right)\right) \right]\left(O(k^2/n)+1\right)\\
        & = k - \left[\left(\sqrt{\frac{k-1}{\pi}}+\frac{1}{2\sqrt{\pi(k-1)}}\right)\left(1 - \frac{1}{8(k-1)} + o\left(\frac{1}{k}\right)\right) \right]\left(O(k^2/n)+1\right)\\
        & = k - \sqrt{\frac{k}{\pi}} + \frac{1}{8\sqrt{\pi k}} + o\left(\frac{1}{\sqrt{k}}\right),
    \end{align*}
    which proves the result.
\end{proof}

\section*{Acknowledgements}
Vijay Subramanian and Hossein Dabirian acknowledge support from NSF via grant CCF-2008130. We are also grateful to Alan Frieze, Bruce Hajek, Remco van der Hofstad, Richard La and Mehrdad Moharrami for helpful comments. Finally, we thank the anonymous reviewers and the associate editor for feedback that greatly improved the readability and presentation of this paper.

\printbibliography

\end{document}